\documentclass[11pt]{article}
\usepackage{xcolor}
\usepackage{etoolbox}
\usepackage{thmtools}
\usepackage{amsmath,amsfonts,amssymb,amsthm}
\usepackage{ucs}
\usepackage{mathptmx}

\AtBeginDocument{
	\addtolength{\abovedisplayskip}{-2ex}
	\addtolength{\abovedisplayshortskip}{-2ex}
	\addtolength{\belowdisplayskip}{-2ex}
	\addtolength{\belowdisplayshortskip}{-2ex}
}
\usepackage[toc,page]{appendix}
\usepackage{setspace}
\usepackage{graphicx}
\graphicspath{ {./images/} }
\usepackage{wrapfig}
\usepackage[usenames,dvipsnames]{pstricks}
\usepackage{fancyhdr}

\usepackage[caption=false]{subfig}
\usepackage{lipsum} 
\usepackage{amsmath,amssymb,amsfonts}

\theoremstyle{break}
\newtheorem{prop}{\textbf{Proposition}}[section]
\newtheorem*{theorem*}{Theorem}
\newtheorem{definition}{\textbf{Definition}}[section]
\newtheorem{theorem}{\textbf{Theorem}}[section]

\newtheorem{claim}{\underline{\textbf{Claim}}}[section]

\newcommand{\R}{\mathbb{R}}

\newcommand{\N}{\mathbb{N}}
\newcommand{\Z}{\mathbb{Z}}

\newcommand{\norm}[1]{\left\lVert#1\right\rVert}

\begin{document}
	\title{Sufficent Conditions for the preservation of Polygonal-Connectedness in an arbitrary normed space}
	\author{Savvas Andronicou and Emmanouil Milakis}
			\date{}

	\maketitle
	\begin{abstract}
		In this article we prove that if $ (X,\norm{\cdot}_X) $ is a normed space and $ U $ is a polygonally-connected subset of $ X $ with $M:=\{S_i:\ i\in I\}\subset \mathcal{P}\left( U\right)  $, a non-empty arbitrary family of discrete, non-empty  subsets of $ U, $ then the property of polygonal-connectedness is also preserved in the resulting set $ U\setminus\left( \bigcup_{i\in I} S_i\right) ,$ under appropriate conditions.
	\end{abstract}
	
		\textbf{Mathematics Subject Classification:} 54D05, 30L99
	
	\textbf{Keywords:} Polygonal-connectedness, arbitrary normed spaces	
	
	\section{Introduction}
	
	The property of polygonal-connectedness appears in many settings from pure applications to abstract theories. For instance in a series of papers by Klee (see \cite{Klee1}, \cite{Klee2}) it played a critical role when one tries to characterise convex sets in different topologies. More recently, in a very interesting result \cite{B}, it was proven that the set of all framelets is polygonally connected  both in $L^2(\R)$ and in multiplicative $L^2(\R,dx/|x|)$ norm, with a generalisation in higher dimensions, answering a fundamental question in the theory of wavelets. Finally one could mention recent results in image processing (see \cite{NPK} and references therein) where  digitization can be defined as a process of transforming a given set in Euclidean space $\R^n$  into a discrete set by considering the Gauss digitization i.e its intersection with $\Z^n$. As it is apparent in \cite{NPK} , due to the approximations induced by the digitization process, the polygonal connectedness of the initial set is not always preserved. 
		
	The main objective of the present article is to investigate certain properties of polygonally connected sets in arbitrary normed spaces. It can be considered as a continuation of paper \cite{AM23} where the authors examined the question of preservation of path-connectedness, in particular, the finding of natural conditions on sets with holes in an arbitrary metric spaces in order to preserve the property of path-connectedness. A similar question stands in an arbitrary normed space for polygonally connected sets. 
	
		 \textit{Question: Assume that $ (X,\norm{\cdot}_X) $ is an arbitrary normed space and $ U $ is a polygonally-connected subset of $ X $. If one removes a family $ M $ of discrete, non-empty  subsets from $U$, is it true that resulting set $ U\setminus \bigcup M $ is also polygonally-connected?}

	The purpose of the present paper is provide the structural condition on sets with holes that allows the property of polygonal-connectedness to remain valid, providing in that way a positive answer to the above mentioned question. Our approach is pure analytic and deals with a delicate construction of polygonal lines which preserve the connectedness of the structure.  To the best of our knowledge, these constructions are absent from the literature. Note that although our intention is to provide an answer to a pure theoretical question, it seems that one will end up constructing similar polygonal lines when treating problems from numerical analysis and computer science, for instance when constructing Algorithms for finding convex hulls of simple polygons in general settings (see for instance \cite{GH}, \cite{McA}, \cite{SVW}). To this perspective, and in addition to the applications in image processing \cite{NPK} mentioned above, we anticipate that our results will also gain attention among scientists working in these areas as well. For other applications in Analysis of PDEs in nodal sets the reader is referred to the introduction of \cite{AM23}.     
	
	The structure of the paper is as follows. Firstly we give the list of notations to be used and in Section 2 we provide the appropriate definitions and present the main results. Section 3 is devoted to the proof of the main theorem.
	
		\section*{Notations}
	Next we will provide a list of symbols and notions to be used in the sequel. Let us consider $ (X,d) $ to be a metric space with  $ A,B\subset X $ and $ x_0\in X. $ We denote \\
	$ B_d\left(x_0,\delta \right):=\{z\in X :\ d(z,x_0)<\delta\}  $ the usual open ball with center $ x\in X $ and radius $ \delta>0 $ \\
	$ int(A):=\{x\in X:\ \exists\delta>0,\ B_d(x,\delta)\subset A\} $, the internal of the set $ A $\\
	$ ext(A):=int(X\setminus A) $, the external of the set $ A $\\
	$ Cl(A)\equiv\bar{A}:=\{x\in X:\ x\ \text{limit point of the set A}\} $, the closure of the set $ A $\\
	$ \partial A:=\bar{A}\cap Cl\left( X\setminus A\right) $, the boundary of the set $ A $\\
	$ iso(A):=\{x\in X:\ \exists \delta_x>0,\ B_d(x,\delta_x)\cap A=\{x\}\} $, the set of isolated points of the set $ A $\\
	\begin{gather}
		dist(x_0,A):=\begin{cases}
			\inf\left\lbrace d(x_0,y):\ y\in A \right\rbrace, & A\neq\emptyset\nonumber\\
			0, & A=\emptyset \nonumber
			\end{cases},\ 
	   dist(A,B):=\begin{cases}
		\inf\left\lbrace d(x,y):\ x\in A,\ x\in B \right\rbrace, & A\neq\emptyset\nonumber\\
		0, & A=\emptyset,\ \text{or}\ B=\emptyset \nonumber
	\end{cases}\\
	\end{gather}
where the quantity $ dist(x_0,A)$ is the distance of point $ x_0 $ from the set  $ A $ and $ dist(A,B) $  is the distance from the set $ A $ to the set $ B. $ \\
	For a continuous function $ \gamma:I\to X $, where $ I $ is an interval in $ \R $,  we use the following symbolism:
	\begin{itemize}
		\item $ \gamma(I):=\{\gamma(t)| \ t\in I\} $ the trace of the curve $ \gamma $ on the interval $ I. $
		\item For $ x,y\in \gamma(I) $, let $ x=\gamma(t_x), y=\gamma(t_y) $ and without loss of generality, let $ t_x<t_y $, then we set: 
		$ \gamma_{xy}:=\{\gamma(t)|\ t\in[t_x,t_y]\} $, i.e the part of curve $ \gamma $ that connects the points $ x $ and $ y. $
	\end{itemize}
	\section{Definitions and Main Result}
	In this section we give some necessary definitions for connected metric spaces and polygonally-connected subsets of normed spaces.
	\begin{definition}[Connected metric space]
		A metric space $ \left( X,d\right)  $ is called \emph{connected} if there is no open partition i.e there is no pair of sets $ \{A,B\} $ such that the following hold:
		\begin{gather}
			(i)\ X=A\cup B,\ (ii) \ A,B\neq\emptyset,\ (iii)\  A\cap B=\emptyset,\ (iv)\ A, B \ \text{are both open}.
		\end{gather}
	\end{definition}
	\begin{definition}[Polygonally-connected set]
		Let $ (X,\norm{\cdot}_X)  $ be a normed space and $ U\subset X. $
		\begin{itemize}
			\item If $ x,y\in X $ then the set $ [x,y]:=\{(1-t)x+ty: t\in[0,1]\} $ is called line segment with endpoints $ x,y. $
			\item If $ (x_j)_{j=1}^{n} $ is a finite sequence of points of $ X, $ (not necessary distinct), then the set $ \bigcup_{j=2}^{n}[x_{j-1},x_j]$ is called polygonal line with initial point $ x_1 $ and final point $ x_n. $ The line seqments $ [x_1,x_2],\ [x_2,x_3],\ \dots,\ [x_{n-1},x_n] $ are called edges of the polygonal line.
			\item For any points $ x,y$ of the space $ X $, with $ x\neq y, $ we say that are connected through a polygonal line, if there exists a polygonal line with initial point $ x $ and final point $ y. $
			\item A set $ U\subset X $ is called polygonally connected, if for every $ x,y\in U, $ there exists a polygonal line $ \bigcup_{j=2}^{n}[x_{j-1},x_j] $ that connects the points $ x,y $ and $ \bigcup_{j=2}^{n}[x_{j-1},x_j]\subset U. $
		\end{itemize}	
		
	\end{definition}
	
	\begin{definition}[U-M property]
		Let $ (X,\norm{\cdot}_X) $ be a normed space and $ U $ a polygonally-connected subset of $ X $ with $M:=\{S_i:\ i\in I\}\subset\mathcal{P}\left( U\right)  $, an arbitraty family of discrete, non-empty subsets of $ U. $ We call that the set $ U $ satisfies the (U-M) property if and only if, for every polygonal line, $ \bigcup_{j=2}^{n}[x_{j-1},x_j] $ the following statement holds:
		\begin{gather}
			\text{For all}\ j\in\{2,\dots,n\},\ \text{the set}\ [x_{j-1},x_j]\cap \left(\bigcup_{i\in I} \partial S_i \right)\ \text{is finite}.
		\end{gather}
	
\end{definition}

Next we are ready to present our main result.

	\begin{theorem}\label{main_theorem_1}
		Let $ (X,\norm{\cdot}_X) $ be a normed space and $ U $ a polygonally-connected subset of $ X $ with $M:=\{S_i:\ i\in I\}\subset\mathcal{P}\left( U\right) $, an non-empty, arbitrary family of discrete\footnote{By discrete family of sets $ M:=\{S_i:\ i\in I\} $, we mean that, the map $ P:I\to M $ is injective i.e for every $ i,j\in I $, if $ i\neq j $ then $ S_i=P(i)\neq P(j)=S_j. $}, non-empty subsets of $ U. $ We assume that:
		\begin{itemize}
			\item[(i)] $ U $ satisfies the $ (U-M) $ property.
			\item [(ii)] For all $ i\in I,\ S_i $ is open, $ \partial S_i\neq\emptyset $ and $ \bar{S}_i\varsubsetneqq int(U) $, with $ \bar{S}_i\cap\bar{S}_j=\emptyset,\ \forall i,j\in I,\ i\neq j $.
			\item[(iii)] For all $ i\in I, $ there exists $ \delta_i>0 $ such that 
			\begin{gather}
				K_{\delta_i}:=\left\lbrace x\in U\setminus\left( \bigcup_{i\in I} S_i\right) :\ dist(x,\ \partial S_i)<\delta_i \right\rbrace 
			\end{gather}
		is polygonally-connected. 
		\end{itemize}
	Then, the resulting set $ U\setminus\left(\bigcup_{i=1}^n S_i\right)  $ is polygonally-connected.
	\end{theorem}

	\section{Proof of Main Result}
	The proof of the Theorem \ref{main_theorem_1} is implemented through the following fundamental to our analysis proposition:
	
	\begin{prop}\label{main_prop}
		Let $ (X,d) $ be an arbitrary metric space, $ K\subset X,$ with $ \partial K\neq\emptyset $ and $\ x_0\in int(K), $ $ y_0\in ext(K) $. If $ \gamma:I\to X $ a continuous function defined on the interval $ I\subset\R $ such that, $ x_0=\gamma(t_{x_0}),y_0=\gamma(t_{y_0}), $ where $ t_{x_0},t_{y_0}\in I, $ then
		\begin{gather}\label{eq_1}
			\gamma(\tilde{I}) \cap\partial K\neq \emptyset			
		\end{gather}
		where $ \tilde{I}:=\begin{cases}
			[t_{x_0},t_{y_0}],\ &t_{x_0}<t_{y_0}\nonumber\\
			[t_{y_0},t_{x_0}],\ &t_{y_0}<t_{x_0}\nonumber.
		\end{cases} $
	\end{prop}
\begin{figure}[h]
	\centering
	\includegraphics[width=50mm]{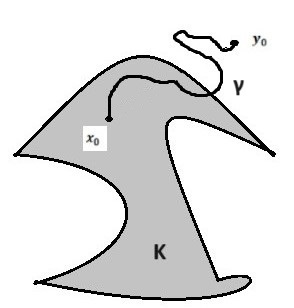}
\end{figure}

	\begin{proof}
		Without loss of generality, let $ t_{x_0}<t_z, $ i.e $ \tilde{I}:=[t_{x_0},t_z]. $
		Let suppose the opposite of what we seek to prove i.e $ \gamma(\tilde{I}) \cap\partial K=\emptyset $. We use the fact that the metric space $ X $ receives a partition  in terms of the set $ K $, i.e
		\begin{gather}
			X=int(K) \cup \partial K\cup ext(K) \nonumber.
		\end{gather}
		As a result from the above, we obtain
		\begin{align}
			\gamma(\tilde{I})&=\gamma(\tilde{I})\cap X \nonumber\\
			&=\left(\gamma(\tilde{I})\cap int(K) \right) \cup \left(\gamma(\tilde{I})\cap \partial K \right) \cup\left( \gamma(\tilde{I})\cap ext(K) \right)\nonumber\\
			&=^{\left(  \gamma(\tilde{I}) \cap\partial K=\emptyset\right)  }\left(\gamma(\tilde{I})\cap int(K) \right)\cup\left( \gamma(\tilde{I})\cap ext(K) \right)\nonumber\\
			&=V_1\cup V_2\nonumber
		\end{align}
		where $ V_1:=\gamma\left( \tilde{I}\right) \cap int(K) $ and $ V_2:=\gamma\left( \tilde{I}\right) \cap ext(K)$. First we see that from the Principle of Inheritance, the sets $ V_1, V_2 $ are open in the metric subspace $ \left(\gamma(\tilde{I}),d_{\gamma(\tilde{I})} \right)  $. Moreover, $ V_1\neq\emptyset, $ since $ x_0\in\gamma(\tilde{I})\cap int(K)  $ and $ V_2\neq\emptyset, $ since from the previous claim, we have $ y_0\in ext(K)  $ and $ y_0\in\gamma\left( \tilde{I}\right) . $ Clearly, we see that $ V_1\cap V_2=\emptyset.$ 
		Consequently, the pair of sets $ \{V_1, V_2\} $ is an open (with respect to the relative metric $ d_{\gamma(\tilde{I})} $) partition of $ \gamma(\tilde{I}) $. From the last we conclude that the set $ \gamma(\tilde{I}) $ is not connected. The latter is contradicted with the fact that, since $ \tilde{I} $ is an interval in $ \R $, therefore a connected set and the function  $ \gamma:\tilde{I}\to X $, is continuous, it follows from known theorem \footnote{Let $ \left(X,d \right),\ \left(Y,\rho \right)   $ metric spaces, and $ f:X\to Y $ be a continuous function. If $ \left(X,d \right) $ is connected space, then the image $ f(X) $ is connected subset of $ Y $.} that $ \gamma(\tilde{I}) $ will be connected. As a result from the above analysis, we receive that:  $  \gamma(\tilde{I})\cap\partial K\neq\emptyset.   $ 
		\end{proof}
Next section is devoted solely to the proof of the main theorem.  As it will be apparent, the proof will be achieved  by distinguishing several appropriate cases.

\begin{figure}[h]
			\centering
			\includegraphics[width=150mm]{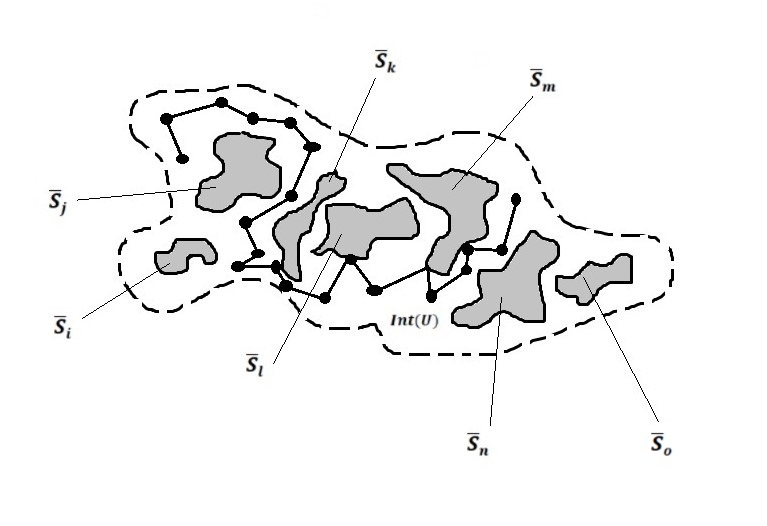}
			\caption{An example that represents the case $ \left( \bigcup_{j=2}^n[x_{j-1},x_j]\right)\bigcap\left(\bigcup_{i\in I}S_i \right)=\emptyset  $ }
		\end{figure}

\subsection{Proof of Theorem \ref{main_theorem_1}}
	\begin{proof}
		Let $ x,y\in U\setminus\left( \bigcup_{i\in I}S_i\right),  $ with $ x\neq y. $ Since by the assumption of the theorem, U is polygonally-connected, there exists a polygonal line inside $ U $ that connects the points $ x $ and $ y. $ In specific, let $ \bigcup_{j=2}^n[x_{j-1},x_j]\subset U $ where $ x_1:=x $ and $ x_n:=y. $ Now we have the following cases:\\
		\textbf{Case:} $ \left( \bigcup_{j=2}^n[x_{j-1},x_j]\right)\bigcap\left(\bigcup_{i\in I}S_i \right)=\emptyset.  $ 
		
		In this case we obtain that $\bigcup_{j=2}^n[x_{j-1},x_j]\subset U\setminus\left(\bigcup_{i\in I}S_i \right).   $ Moreover, observe that it is possible to have the situation $ \left( \bigcup_{j=2}^n[x_{j-1},x_j]\right) \cap \left( \bigcup_{i\in I}\partial S_i\right)\neq\emptyset.  $ Indeed, from the assumption of the case, we receive that $ \bigcup_{j=2}^n[x_{j-1},x_j]\subset \bigcup_{i\in I} S^c_i $. Moreover, since $ S_i $ is open, it folows that $ S_i\cap \partial S_i=\emptyset,\ \forall i\in I. $ Equivalently, $ \partial S_{\tilde{i}}\subset S^c_{\tilde{i}}\subset\bigcup_{i\in I} S^c_i,\ \forall \tilde{i}\in I.  $ Therefore, $ \bigcup_{i\in I}\partial S_i\subset\bigcup_{i\in I} S^c_i. $\\
		\textbf{Case:} $ \left( \bigcup_{j=2}^n[x_{j-1},x_j]\right)\bigcap\left(\bigcup_{i\in I}S_i \right)\neq\emptyset.  $\\
		 We define 
		 \begin{gather}
		 	J:=\left\lbrace j\in\{2,3,\dots,n\}:\ \exists i\in I,\  [x_{j-1},x_j]\cap \partial S_i \neq\emptyset\ \text{and}\  [x_{j-1},x_j]\cap S_i\neq\emptyset\right\rbrace. 
		 \end{gather}
	 \begin{claim}
	 	The set $ J $ is not empty.
	 \end{claim}
 \begin{proof}
 	By our consideration above, there exists $ z\in \bigcup_{j=2}^n[x_{j-1},x_j]$ and $ z\in\bigcup_{i\in I}S_i. $ Equivalently, there exists $ j^*\in\{2,3,\dots,n\} $ and $ i_0\in I $ such that, $ z\in[x_{j^*-1},x_{j^*}] $ and $ z\in S_{i_0}. $ Then, we distinguish the following cases:
 	\begin{itemize}
 		\item $[x_{j^*-1},x_{j^*}]\cap\partial S_{i_0}\neq\emptyset.  $\\
 		If this case holds, then we obviously have that $ j^*\in J $
 		\item $ [x_{j^*-1},x_{j^*}]\cap\partial S_{i_0}=\emptyset $.\\
 		If this case holds then we obtain that, $ [x_{j^*-1},x_{j^*}]\subset int(S_{i_0})=^{(S_{i_0}\ \text{open})}S_{i_0}. $ Indeed, we observe that,
 		\begin{align}
 			[x_{j^*-1},x_{j^*}]&=\left( [x_{j^*-1},x_{j^*}]\cap int(S_{i_0})\right) \cup \left( [x_{j^*-1},x_{j^*}]\cap \partial S_{i_0}\right) \cup \left([x_{j^*-1},x_{j^*}]\cap ext( S_{i_0}) \right) \nonumber\\
 			&=\left( [x_{j^*-1},x_{j^*}]\cap int(S_{i_0})\right) \cup \left([x_{j^*-1},x_{j^*}]\cap ext( S_{i_0}) \right)=V_1\cup V_2
 		\end{align} 
 	where $ V_1, V_2 $ are open in the metric subspace $ [x_{j^*-1},x_{j^*}] $ with $ V_1\cap V_2=\emptyset. $ Because $ V_1\neq\emptyset, $ (since $z\in V_1  $), and $[x_{j^*-1},x_{j^*}]  $ is connected as a continuous image of the interval $ [0,1], $ we conclude that $ V_2=\emptyset. $ Therefore $ [x_{j^*-1},x_{j^*}]=V_1\subset int(S_{i_0})=S_{i_0}. $\\
 	Next, we define
 	\begin{gather}
 		L:=\{j\in\{j^*,j^*+1,\dots,n-1\}:\ [x_j,x_{j+1}]\cap\partial S_{i_0}\neq\emptyset\}.
 	\end{gather}
 	We claim that, $ L\neq\emptyset. $ Indeed, by the assumption of the theorem, $ y\in U\setminus\left( \bigcup_{i\in I}S_i\right) $, and as a result from the last, we obtain that $ y\equiv x_n\notin S_{i_0}. $ Therefore, since  $ S_{i_0} $ is open, it follows that $ y\in\partial S_{i_0}\cup int(S^c_{i_0}). $ Then, we distinguish the following cases:
 	\begin{itemize}
 		\item $ y\in\partial S_{i_0}. $ \\In this case, since $ y\equiv x_n\in[x_{n-1},x_n], $ by the definition of $ L, $ we receive $ n-1\in L. $
 		\item $ y\in int(S^c_{i_0})\equiv ext(S_{i_0}). $ \\In this case we define the path, $ \gamma_{z,y}:[a,b]\to X $ where $ \gamma_{z,y}:=\gamma_{[z,x_{j^*}]}\oplus\left( \oplus_{j=j^*}^{n-1}\gamma_{[x_{j},x_{j+1}]} \right) $. Obviously the path $ \gamma_{z,y} $ is continuous and connects the points $z\in S_{i_0}=int (S_{i_0}) $ and $ y\equiv x_n\in int(S^c_{i_0}) $.  Finally, from the Proposition \ref{main_prop}, it holds that $ \gamma_{z,y}([a,b]) \cap\partial S_{i_0}\neq\emptyset.$ Therefore, there exists $ t_0\in[a,b] $ such that $ \gamma_{z,y}(t_0)\in\partial S_{i_0}. $ Equivalently, there exists $ \tilde{j}\in\{j^*,\dots,n-1\} $ such that $ \gamma_{z,y}(t_0)\in[x_{\tilde{j}},x_{\tilde{j}+1}]\cap\partial S_{i_0}. $ From the last, we conclude that $ \tilde{j}\in L. $
 	\end{itemize}
 From both cases, we get that $ L\neq\emptyset $ and we are in the position to define $ \bar{j}:=\min L $. By the definition of $ \bar{j}, $ it follows that 
 \begin{gather}\label{rel_1}
 	[x_{\bar{j}},x_{\bar{j}+1}]\cap\partial S_{i_0}\neq\emptyset\ \text{and}\
 	\forall j\in\{j^*,\bar{j}-1\},\ [x_j,x_{j+1}]\cap\partial S_{i_0}=\emptyset.
 \end{gather} 
We claim now that, $ [x_{\bar{j}},x_{\bar{j}+1}]\cap S_{i_0}\neq\emptyset $. Indeed, let us suppose that $[x_{\bar{j}},x_{\bar{j}+1}]\cap S_{i_0}=\emptyset.  $ We define
\begin{align}
	\tilde{t}&:=\inf\{t\in[0,1]:\ (1-t)x_{\bar{j}}+t x_{\bar{j}+1}\in\partial S_{i_0}\}\equiv\inf L_{x_{\bar{j}}, x_{\bar{j}+1}}\nonumber\\
	\tilde{z}&:=(1-\tilde{t})x_{\bar{j}}+\tilde{t}x_{\bar{j}+1}.
\end{align}
We notice that, $ L_{x_{\bar{j}}, x_{\bar{j}+1}}:=\{t\in[0,1]:\ (1-t)x_{\bar{j}}+t x_{\bar{j}+1}\in\partial S_{i_0}\} $ is a non-empty set since $ [x_{\bar{j}},x_{\bar{j}+1}]\cap\partial S_{i_0}\neq\emptyset. $ Moreover, $ L_{x_{\bar{j}}, x_{\bar{j}+1}} $ is lower bounded from $ 0, $ therefore $ \tilde{t} $ is well defined. Now we claim that $ \tilde{t}\in[0,1]. $ Indeed, since $ [x_{\bar{j}},x_{\bar{j}+1}]\cap\partial S_{i_0}\neq\emptyset  $, there exists $ t^*\in L_{x_{\bar{j}}, x_{\bar{j}+1}}  $ and by the definition of $ \tilde{t}, $ we receive that $0\leq\tilde{t}\leq t^*\leq 1.  $

 Furthermore, $ \tilde{z}\in\partial S_{i_0}. $ Indeed, since $ \tilde{t} $ is a limit point of $ L_{x_{\bar{j}}, x_{\bar{j}+1}}, $ there exists a sequence $ (t_k)_{k\in\N}\subset L_{x_{\bar{j}}, x_{\bar{j}+1}}  $ such that, $ t_k\xrightarrow{k\to\infty}\tilde{t}. $ Now we define  the map $ s:[0,1]\to X,  $ with $ s(t):=(1-t)x_{\bar{j}}+t x_{\bar{j}+1}. $ From the continouity of $ s $ it follows that, $ s(t_k)\xrightarrow{k\to\infty}s(\tilde{t})\equiv\tilde{z}. $ Since, $ t_k\in L_{x_{\bar{j}}, x_{\bar{j}+1}},\ \forall k\in\N, $ it follows that $ s(t_k)\in\partial S_{i_0}. $ Therefore, $ \tilde{z} $ is a limit point of $ \partial S_{i_0} $ and due to the fact that $ \partial S_{i_0} $ is closed, we obtain finally that $ \tilde{z}\in\partial S_{i_0}. $\\
We distinguish the following cases for $ \tilde{t}. $ In specific, if $ \tilde{t}=0, $ then $ \tilde{z}=x_{\bar{j}}. $ Therefore, $ x_{\bar{j}}\in\partial S_{i_0} $ and as a result from the last, we receive that $ [x_{\bar{j}-1},x_{\bar{j}}]\cap\partial S_{i_0}\neq\emptyset $, which is a contradiction from the second part of (\ref{rel_1}). Now, if $ \tilde{t}\in (0,1], $ then from the definition of $ \tilde{t}, $ we obtain that for all  $ t\in[0,\tilde{t}),\ z_t:=(1-t)x_{\bar{j}}+ t x_{\bar{j}+1}\notin\partial S_{i_0}. $ Therefore, $ z_t\in ext(S_{i_0}),\ \forall t\in[0,\tilde{t}). $ (since $ [x_{\bar{j}},x_{\bar{j}+1}]\cap S_{i_0}=\emptyset $). For $ t=0, $ we receive that, $ z_0=x_{\bar{j}} $ and from the last we conclude that $ x_{\bar{j}}\in ext(S_{i_0}). $ At the same time, we know from the initial part of the proof that $ z\in[x_{j^*-1},x_{j^*}]\subset S_{i_0}. $ Next, we define the path, $ \gamma_{z,x_{\bar{j}}}:[\tilde{a},\tilde{b}]\to X$ where $ \gamma_{z,x_{\bar{j}}}:=\gamma_{[z,x_{j^*}]}\oplus\left(\oplus_{j=j^*}^{\bar{j}-1} \gamma_{[x_j,x_{j+1}]} \right).  $ Obviously, the path $ \gamma_{z,x_{\bar{j}}} $ is continuous and connects the points $ z\in S_{i_0} $ and $ x_{\bar{j}}\in ext(S_{i_0}). $ Finally from Proposition \ref{main_prop}, it holds that $ \gamma_{z,x_{\bar{j}}}([\tilde{a},\tilde{b}]) \cap\partial S_{i_0}\neq\emptyset$. Therefore, there exists $ \tilde{t_0}\in[\tilde{a},\tilde{b}] $ such that $ \gamma_{z,x_{\bar{j}}}(\tilde{t_0})\in \partial S_{i_0}. $ Equivalently, there exists $ \tilde{j}\in\{j^*,\dots,\bar{j}-1\},$ ($ [x_{j^*-1},x_{j^*}]\subset S_{i_0} $ and $ S_{i_0}\cap\partial S_{i_0}\neq\emptyset, $ since $ S_{i_0} $ is open ) such that  $ \gamma_{z,x_{\bar{j}}}(\tilde{t_0})\in[x_{\tilde{j}},x_{\tilde{j}+1}]\cap\partial S_{i_0}. $ I.e $ \tilde{j}\in L $ and $ \tilde{j}<\bar{j} $. The last contradicts with the definition of $ \bar{j}:=\min L. $\\
From the above analysis we conclude that $ [x_{\bar{j}},x_{\bar{j}+1}]\cap S_{i_0}\neq\emptyset. $ Then by the definition of $ \bar{j} $ it follows that $[x_{\bar{j}},x_{\bar{j}+1}]\cap \partial S_{i_0} \neq\emptyset, $ therefore by the definition of the set $ J, $ we obtain $ \bar{j}+1\in J. $
\end{itemize}
\end{proof}

Next, we define the following indexes
\begin{gather}
	k_1:=\min J,\ k_m:=\min(J\setminus\{k_{m-1}\}),\ \forall m\in\{2,3,\dots,j_0\}\ (\text{for}\ j_0\geq 2)
\end{gather}
where $ j_0:=\max J. $ I.e for $ j_0\geq 2,\ J=\{k_1<k_2<\dots<k_{j_0}\}. $ 
\begin{claim}
	If $ k_1\geq 3, $ then $\forall i\in I,\ \forall l\in\{2,3,\dots,k_1-1\},\  [x_{l-1},x_l]\subset\partial S_i \cup ext(S_i) $.
\end{claim}
\begin{proof}
	For every $ l\in\{2,3,\dots,k_1-1\} $ it is obvious that $ l\notin J. $ Let $ l_0\in\{2,3,\dots,k_1-1\} $ be fixed, randomly selected. Therefore, by the definition of the set $ J, $ we conclude that, for all $  i\in I,  $ the conditions $ [x_{l_0-1},x_{l_0}] \cap \partial S_i\neq\emptyset$ and $ [x_{l_0-1},x_{l_0}] \cap S_i\neq\emptyset $ cannot hold at the same time. The last is equivalent to the following: $ \forall i\in I, $
	\begin{gather}
		[x_{l_0-1},x_{l_0}] \cap \partial S_i=\emptyset\ \text{and}\ [x_{l_0-1},x_{l_0}] \cap  S_i=\emptyset\label{eq_1}\\
		\text{or}\ [x_{l_0-1},x_{l_0}] \cap \partial S_i\neq\emptyset\ \text{and}\ [x_{l_0-1},x_{l_0}] \cap  S_i=\emptyset\label{eq_2}\\
		\text{or}\ [x_{l_0-1},x_{l_0}] \cap \partial S_i=\emptyset\ \text{and}\ [x_{l_0-1},x_{l_0}] \cap  S_i\neq\emptyset.\label{eq_3}
	\end{gather}
	For a random $ i\in I $, if the condition $ (\ref{eq_1})  $ holds, we receive that, $ [x_{l_0-1},x_{l_0}] \subset ext(S_i)\subset \partial S_i\cup ext(S_i)$. Moreover, for a random $ i\in I $ if the condition $ (\ref{eq_2})  $ holds, we obtain that  $ [x_{l_0-1},x_{l_0}] \subset \partial S_i\cup ext(S_i)$. Now we will show that the condition $ (\ref{eq_3})  $ cannot be true. Indeed,  let that condition to be true i.e we allow the case where for a random $ \tilde{i}\in I $ the condition $ (\ref{eq_3}) $ holds. Then using the fact that the line interval $ [x_{l_0-1},x_{l_0}] $ is connected, we obtain $ [x_{l_0-1},x_{l_0}] \subset S_{\tilde{i}}.$ As a result there exists $ z\in[x_{l_0-1},x_{l_0}] , z\in int(S_{\tilde{i}})$. Also we notice that $ x_1=x\notin S_{\tilde{i}}, $ i.e $ x_1\in\partial S_{\tilde{i}}\cup ext(S_{\tilde{i}}) $. We consider the polygonal line $ \gamma^*_{x_1,z}:=\cup_{q=2}^{l_0-1}\gamma_{[x_{q-1},x_q]}\cup\gamma_{[x_{l_0-1},z]} $. Since the map $ \gamma^*_{x_1,z} $ is continuous, using the Proposition $ \ref{main_prop} $, we receive that $ \gamma^*(I) \cap\partial S_{\tilde{i}}\neq\emptyset.$ Therefore, there exists $ q_0\in\{2,3,\dots,l_0-1\} $ such that $ [x_{q_0-1},x_{q_0}]\cap\partial S_{\tilde{i}}\neq\emptyset.  $ We show that the last statement cannot be true. In specific, we prove that $ \forall q\in\{2,\dots,l_0-1\},\ [x_{q-1},x_q]\cap\partial S_{\tilde{i}}=\emptyset $. Let us suppose that $ q_0=l_0-1 $ i.e $ [x_{l_0-2},x_{l_0-1}]\cap\partial S_{\tilde{i}}\neq\emptyset.$ We know already that $ x_{l_0-1}\in S_{\tilde{i}} $ (since $ [x_{l_0-1},x_{l_0}]\subset S_{\tilde{i}} $). From the last we conclude by the definition of the set $ J, $ that $ l_0-1\in J, $ therefore $ k_1=\min J\leq l_0-1. $ This is a contradiction, since $ l_0\leq k_1-1<k_1 $ and as a result  $ l_0-1<k_1 .$ Furthermore, we obtain that $ [x_{l_0-2},x_{l_0-1}]\cap\partial S_{\tilde{i}} =\emptyset$. Thus, $ [x_{l_0-2},x_{l_0-1}]\subset S_{\tilde{i}}\cup ext(S_{\tilde{i}}). $ Again using the connectedness argument of the line interval $ [x_{l_0-2},x_{l_0-1}] $ and the fact that $ [x_{l_0-2},x_{l_0-1}]\cap S_{\tilde{i}}\neq\emptyset $, we receive that $ [x_{l_0-2},x_{l_0-1}]\subset S_{\tilde{i}} $ hence, $ q_0\in\{2,\dots,l_0-2\}. $ If we suppose now that $ q_0=l_0-2 $, repeating the previous procedure, we conclude that $ [x_{l_0-3},x_{l_0-2}]\subset S_{\tilde{i}}. $ Continuing this procedure, until $ q_0=2, $ we conclude that $ [x_1,x_2]\subset S_{\tilde{i}} $ which is a contradiction since $ x\equiv x_1\in[x_1,x_2]\subset S_{\tilde{i}} $ and $ x\notin\cup_{i\in I}S_i. $ Thus, we proved that $ \forall q\in\{2,\dots,l_0-1\},\ [x_{q-1},x_q]\cap\partial S_{\tilde{i}}=\emptyset. $ In consequence, the condition $ (\ref{eq_3}) $ does not hold. Hence, $ \forall i\in I,\ [x_{l_0-1},x_{l_0}]\subset\partial S_i\cup ext(S_i). $
\end{proof}

For each $ m\in\{1,2,\dots,j_0\} $ we define:
\begin{gather}
	Q_m:=\left\lbrace i\in I:\ [x_{k_m-1},x_{k_m}]\cap\partial S_i\neq\emptyset\ \text{and}\ [x_{k_m-1},x_{k_m}]\cap S_i\neq\emptyset\right\rbrace. 
\end{gather}
\begin{claim}\label{cl_imp_1}
	For all $ m\in\{1,2,\dots,j_0\} $ the set $ Q_m $ is finite, non-empty.
\end{claim}
\begin{proof}
	Let $ m\in\{1,2,\dots,j_0\} $ be randomly selected. Initially we prove that $ Q_m\neq\emptyset. $ By the definition of $ k_m $, there exists $ i^*\in I $ such that $ [x_{k_m-1},x_{k_m}]\cap S_{i^*}\neq\emptyset. $ Therefore, $ i^*\in Q_m. $ Next, we prove that, $ Q_m $ is finite. Let us suppose the opposite of what we seek to prove i.e $ Q_m $ is infinite. Consequetly there exists countably infinite  $ P\subset Q_m $. Let $ P:=\{i_{\lambda}:\ \lambda\in\N\} $. Since, $ P\subset Q_m, $ for each $ \lambda\in\N $ there exists $ y_{\lambda}\in[x_{k_m-1},x_{k_m}]\cap\partial S_{i_{\lambda}} $. We claim that the sequence $ (y_{\lambda})_{\lambda\in\N} $ consists of distinct terms. Indeed, let $ \lambda_1,\lambda_2\in\N $ with $ \lambda_1\neq \lambda_2 $ such that $ y_{\lambda_1}=y_{\lambda_2} $. As a result from the last, we receive $ \partial S_{i_{\lambda_1}}\cap \partial S_{i_{\lambda_2}}\neq\emptyset, $ which contradicts with the assumption of the theorem, that $ \bar{S}_{i_{\lambda_1}}\cap \bar{S}_{i_{\lambda_2}}=\emptyset. $ We conclude then, that $ 	\{y_{i_{\lambda}}:\ \lambda\in\N\} $ is countably infinite. From the fact that
	\begin{gather}
		\{y_{i_{\lambda}}:\ \lambda\in\N\}\subset[x_{k_m-1},x_{k_m}]\cap\left( \cup_{\lambda\in\N}\partial S_{i_{\lambda}} \right)\subset [x_{k_m-1},x_{k_m}]\cap\left( \cup_{i\in I}\partial S_{i} \right)
	\end{gather}
we obtain that the set $ [x_{k_m-1},x_{k_m}]\cap\left( \cup_{i\in I}\partial S_{i} \right)  $ is infinite, where it contradicts with the $ (U-M) $ assumption of the theorem.
  \end{proof}
For all $ m\in\{1,2,\dots,j_0\},\ i\in Q_m $ we define:
\begin{gather}
	L^{(m)}_{i}:=\{t\in[0,1]:\ z^{(m)}(t):=(1-t){x_{k_m-1}}+t x_{k_m}\in\partial S_i\}\nonumber\\
	t^{(m)}_{i,min}:=\inf L^{(m)}_{i}\ \text{and}\ t^{(m)}_{i,max}:=\sup L^{(m)}_{i}.
\end{gather}

\begin{claim}\label{cl_imp_2}
For every $ m\in\{1,2,\dots,j_0\} $ the following hold:
	\begin{itemize}
		\item [(i)] For all $  i\in Q_m,\ L^{(m)}_{i}\neq\emptyset, $ finite. Let $ L^{(m)}_{i}:=\{t^{(m)}_{i,1}<t^{(m)}_{i,2}<\dots<t^{(m)}_{i,r_{m,i}}\} $ where $ r_{m,i}:=card\left(L^{(m)}_{i} \right)  $
		\item [(ii)] For all $ i,j\in Q_m,\ \text{with}\  i\neq j,\  L^{(m)}_{i}\cap L^{(m)}_{j}=\emptyset$. 
	\end{itemize}
\end{claim}
\begin{proof}
	Let  $ m\in\{1,2,\dots,j_0\} $ be randomly selected.
	\begin{itemize}
		\item[(i)] Let $  i\in Q_m,$ be randomly selected. Therefore, $ [x_{k_m-1},x_{k_m}]\cap S_i\neq\emptyset $ and $ [x_{k_m-1},x_{k_m}]\cap\partial S_i\neq\emptyset. $ From the last, it follows that there exists $ y\in[x_{k_m-1},x_{k_m}] $ and $ y\in\partial S_i. $ Equivalently, there exists $ \tilde{t}\in[0,1], y=z^{(m)}(\tilde{t})\in\partial S_i. $ Hence, $ \tilde{t}\in L^{(m)}_i. $\\
		We proceed now to prove that for all $ i\in Q_m $ the set $ L^{(m)}_{i} $ is finite. Let $  i^*\in Q_m,$ be randomly selected. From the definition of $ L^{(m)}_{i^*} $ it holds that $\forall t\in  L^{(m)}_{i^*},\ z^{(m)}(t)\in\partial S_{i^*}. $ Therefore, $ z^{(m)}(L^{(m)}_{i^*})\subset [x_{k_m-1},x_{k_m}]\cap\partial S_{i^*}\subset [x_{k_m-1},x_{k_m}]\cap \left(\cup_{i\in I}\partial S_{i} \right) $. From the $ (U-M) $ assumption, we have that $ [x_{k_m-1},x_{k_m}]\cap \left(\cup_{i\in I}\partial S_{i} \right) $ is finite, therefore $ z^{(m)}(L^{(m)}_{i^*}) $ is finite. Since, $ z^{(m)} $ is injective, we conclude that $ L^{(m)}_{i^*} $ is finite.
		\item [(ii)]  Let us suppose the opposite of what we seek to prove i.e there exist $ i,j\in Q_m $ with $ i\neq j $ such that $ L^{(m)}_i\cap L^{(m)}_j\neq\emptyset.  $ Therefore, there exists $ t_0\in  L^{(m)}_i  $ and $ t_0\in  L^{(m)}_j  $. Thus, $ z^{(m)}(t_0)\in\partial S_i $ and $ z^{(m)}(t_0)\in\partial S_j $. From the last, we obtain, $ \partial S_i\cap\partial S_j\neq\emptyset $ which contradicts with the assumption of theorem, $ \bar{S}_i \cap \bar{S}_j=\emptyset. $ 
	\end{itemize}
\end{proof}
Next, we define for every $  m\in\{1,2,\dots,j_0\} $ a special type of  sequences: $ \left( R_{i_{(v,m)}}\right)_{v\in\N}\subset [0,1],\ \left(i_{(v,m)} \right)_{v\in\N}\subset Q_m,\ \left( t^{(m)}_{i_{(v,m)}}\right)_{v\in\N}\subset\bigcup_{v\in\N} R_{i_{(v,m)}}    $. In specific, for every $  m\in\{1,2,\dots,j_0\} $ we define by induction the following sequences:
\begin{equation}
	 i_{(1,m)}\in Q_m\ \text{such that}\ t^{(m)}_{i_{(1,m)}}:=\min\{t^{(m)}_{i,min}:\ i\in Q_m\}\equiv t^{(m)}_{i_{(1,m)},min}\nonumber
	\end{equation}
\newline
and
	 \begin{equation}
	 \begin{cases}
	 	R_{i_{(1,m)}}:=\left\lbrace t\in\bigcup_{q\in Q_m\setminus\{i_{(1,m)}\}} L^{(m)}_{q}\ :\ \exists\tilde{\delta}_t>0,\ z^{(m)}(t+\tilde{\delta})\in \bigcup_{q\in Q_m\setminus\{i_{(1,m)}\}}S_{q}\ \text{and}\ t>t^{(m)}_{i_{(1,m)},\max}\right\rbrace\nonumber\\
	 	i_{(2,m)}:=\begin{cases}
	 		\arg\left\lbrace q\in Q_m\setminus\{i_{(1,m)}\}:\exists s\in\{1,2,\dots,r_{m,q}\}\ t^{(m)}_{q,s}=\min R_{i_{(1,m)}}\right\rbrace ,	& R_{i_{(1,m)}}\neq\emptyset\nonumber\\
	 		i_{(1,m)},	& R_{i_{(1,m)}}=\emptyset\nonumber
	 		\end{cases}
	 		\\
	 		t^{(m)}_{i_{(2,m)}}:=\begin{cases}
	 			\min R_{i_{(1,m)}}, & R_{i_{(1,m)}}\neq\emptyset\nonumber\\
	 			t^{(m)}_{i_{(1,m)},\max}, & R_{i_{(1,m)}}=\emptyset\nonumber\\
	 		\end{cases}
	 	\end{cases}
	 \end{equation}
 \newline
	 	For all  $ v\geq 3 $,
	 	\begin{equation}
	 	\begin{cases}
	 		R_{i_{(v-1,m)}}:=\left\lbrace t\in\bigcup_{q\in Q_m\setminus\{i_{(1,m)},\dots,i_{(v-1,m)}\}} L^{(m)}_{q}\ :\ \exists\tilde{\delta}_t>0,\ z^{(m)}(t+\tilde{\delta})\in \bigcup_{q\in Q_m\setminus\{i_{(1,m)},\dots,i_{(v-1,m)}\}}S_{q}\ \text{and}\ t>t^{(m)}_{i_{(v-1,m)},\max}\right\rbrace\nonumber\\
	 		i_{(v,m)}:=\begin{cases}
	 			\arg\left\lbrace q\in Q_m\setminus\{i_{(1,m)},\dots,i_{(v-1,m)}\}:\exists s\in\{1,2,\dots,r_{m,q}\}\ t^{(m)}_{q,s}=\min R_{i_{(v-1,m)}}\right\rbrace ,	& R_{i_{(v-1,m)}}\neq\emptyset\nonumber\\
	 			i_{(v-1,m)},	& R_{i_{(v-1,m)}}=\emptyset\nonumber
	 			\end{cases}
	 			\\
	 			t^{(m)}_{i_{(v,m)}}:=\begin{cases}
	 				\min R_{i_{(v-1,m)}}, & R_{i_{(v-1,m)}}\neq\emptyset\nonumber\\
	 				t^{(m)}_{i_{(v-1,m)},\max}, & R_{i_{(v-1,m)}}=\emptyset\nonumber
	 			\end{cases}
	 		\end{cases}
	\end{equation}
\\
\newline
	The next Claim guarantees that the above recursive definition of sequences is well-defined.
\begin{claim}\label{cl_imp_3}
	For every $ m\in\{1,2,\dots,j_0\} $ the following hold:
	\begin{itemize} 
		\item [(i)] For every $ v\in\N $ the set $ R_{i_{(v,m)}} $ is finite
		\item[(ii)] For every $ v\in\N $ with $ v\geq 2, $ the sequence of indexes $ \left( i_{(v,m)}\right) _{v\geq2}  $ is well defined i.e For all $ v\in\N $ with $ v\geq 2, $ if  $ R_{i_{(v-1,m)}}\neq\emptyset $, then there exists unique $ q\in Q_m\setminus\{i_{(1,m)},\dots,i_{(v-1,m)}\} $ with the property: $ \exists s\in\{1,2,\dots,r_{m,q}\}\ t^{(m)}_{q,s}=\min R_{i_{(v-1,m)}} $
		\item[(iii)] For every $ k,\lambda\in\N, $ with $ k,\lambda\geq 2 $, if $ k\neq\lambda $ and $ R_{i_{(k-1,m)}},R_{i_{(\lambda-1,m)}}\neq\emptyset $ then $ i_{(k,m)}\neq i_{(\lambda,m)} $
		\item[(iv)] If there exists $ v_0\in\N $ such that $ R_{i_{(v_0,m)}}=\emptyset $ then for all $ v\geq v_0,\ R_{i_{(v,m)}}=\emptyset $
		\item[(vi)] For every $ k,\lambda\in\N $ with $ k,\lambda\geq 2, $ if $ k<\lambda $ and $ R_{i_{(k-1,m)}},R_{i_{(\lambda-1,m)}}\neq\emptyset $ then $ t^{(m)}_{i_{(k,m)}} < t^{(m)}_{i_{(\lambda,m)}}  $
		\item[(vii)] There exists $ v\in\N $ with $ v\geq 2 $ such that $ R_{i_{(v-1,m)}}=\emptyset $
	\end{itemize}
\end{claim}
	\begin{proof}
		Let $ m\in\{1,2,\dots,j_0\} $ be randomly selected.
		\begin{itemize}
			\item[(i)] For all $ v\in\N $, by the definition of $ R_{i_{(v,m)}} $ we have
			\begin{gather}
				R_{i_{(v,m)}}\subset\bigcup_{q\in Q_m\setminus\{i_{(1,m)},\dots,i_{(v,m)}\}}L^{(m)}_{q}\subset\bigcup_{q\in Q_m}L^{(m)}_{q}\nonumber
			\end{gather}
		where from the Claim \ref{cl_imp_2} (i), it holds that $ L^{(m)}_{q} $ is finite for every $ q\in Q_m. $ Hence, $ \bigcup_{q\in Q_m}L^{(m)}_{q} $ is finite, which follows that $ R_{i_{(v,m)}} $ is also finite.
		
			\item[(ii)] Let $ v\in\N $ with $ v\geq 2 $ such that $ R_{i_{(v-1,m)}}\neq\emptyset $ and let $ q_1,q_2\in Q_m\setminus\{i_{(1,m)},\dots,i_{(v-1,m)}\} $ with $ q_1\neq q_2 $ which satisfy the following:
			\begin{gather}
				\exists s_1\in\{1,2,\dots,r_{m,q_1}\}, t^{(m)}_{\lambda_1,s_1}:=\min R_{i_{(v-1,m)}}\nonumber\\
				\exists s_2\in\{1,2,\dots,r_{m,q_2}\}, t^{(m)}_{\lambda_2,s_2}:=\min R_{i_{(v-1,m)}}\nonumber
			\end{gather}
			By the definition of numbers $ t^{(m)}_{q_1,s_1},t^{(m)}_{q_2,s_2} $, we obtain that $  t^{(m)}_{q_1,s_1}\in L^{(m)}_{q_1} $ and $ t^{(m)}_{q_2,s_2}\in L^{(m)}_{q_2}.  $ Therefore, $ L^{(m)}_{q_1} \cap L^{(m)}_{q_2}\neq\emptyset, $ which leads to a contradiction with the fact that since $ q_1\neq q_2, $ from Claim \ref{cl_imp_2}(ii) it holds that, $ L^{(m)}_{q_1} \cap L^{(m)}_{q_2}=\emptyset. $ 
			
			\item[(iii)] Let $ k,\lambda\in\N, $ with $ k,\lambda\geq 2 $ and $ k<\lambda $ such that $ R_{i_{(k-1,m)}},R_{i_{(\lambda-1,m)}}\neq\emptyset.  $ Therefore, the indexes $ i_{(k,m)}, i_{(\lambda,m)} $ are well defined. Thus by the definition of index $ i_{(\lambda,m)} $, it follows that\\ $ i_{(\lambda,m)}\in Q_m\setminus\{i_{(1,m)},i_{(2,m)},\dots,i_{(k,m)},\dots,i_{(\lambda-1,m)}\}. $ Consequently, $ i_{(\lambda,m)}\neq i_{(k,m)}.  $
			
			\item[(iv)] We prove this by induction. From the assumption we have that $ R_{i_{(v_0,m)}}=\emptyset. $ Let $ r \geq v_0 $ such that $ R_{i_{(r,m)}}=\emptyset.  $ We claim that $ R_{i_{(r+1,m)}}=\emptyset. $ Indeed, let suppose that $ R_{i_{(r+1,m)}}\neq\emptyset. $ Therefore, there exists $ t_0\in \bigcup_{q\in Q_m\setminus\{i_{(1,m)},\dots,i_{(r,m)},i_{(r+1,m)}\}} L^{(m)}_q $ such that the following hold:
			\begin{gather}
				 t_0>t^{(m)}_{i_{(r+1,m)},\max}\ \text{and}\ z^{(m)}(t_0+\delta_{t_0})\in\bigcup_{q\in Q_m\setminus\{i_{(1,m)},\dots,i_{(r,m)},i_{(r+1,m)}\}} S_q
			\end{gather}
			 We observe that $ \bigcup_{q\in Q_m\setminus\{i_{(1,m)},\dots,i_{(r,m)}i_{(r+1,m)}\}} S_q\subset \bigcup_{q\in Q_m\setminus\{i_{(1,m)},\dots,i_{(r,m)}\}} S_q.  $ Moreover, since $ R_{i_{(r,m)}}=\emptyset $, from the definition of $ t^{(m)}_{i_{(r+1,m)}} $, it holds that $ t^{(m)}_{i_{(r+1,m)}}=t^{(m)}_{i_{(r,m)},\max}. $ Hence, $ t_0>t^{(m)}_{i_{(r+1,m)},\max}\geq t^{(m)}_{i_{(r+1,m)}}= t^{(m)}_{i_{(r,m)},\max}. $ I.e $ t_0\in R_{i_{(r,m)}}  $ which contradicts with the assumption of the induction step, $ R_{i_{(r,m)}}=\emptyset. $ Consequently, we proved that $ R_{i_{(r+1,m)}}=\emptyset $ and by the Induction we get the desired result.
			
			\item [(vi)] Since $ R_{i_{(\lambda-1,m)}}\neq\emptyset $, by the definition of $ t^{(m)}_{i_{(\lambda,m)}} $ we receive that $ t^{(m)}_{i_{(\lambda,m)}}=\min R_{i_{(\lambda-1,m)}}.  $ Moreover, from the definition of $ R_{i_{(\lambda-1,m)}} $, we obtain that for all $ t\in R_{i_{(\lambda-1,m)}},\ t>t^{(m)}_{i_{(\lambda-1,m)},\max}.  $ Therefore,
			\begin{gather}\label{eq_4}
				R_{i_{(\lambda-1,m)}}\ni t^{(m)}_{i_{(\lambda,m)}}> t^{(m)}_{i_{(\lambda-1,m)},\max}\geq  t^{(m)}_{i_{(\lambda-1,m)}}
			\end{gather} 
			 Since, $ R_{i_{(\lambda-1,m)}}\neq\emptyset  $ then $ \forall v\in\{k,k+1,\dots,\lambda-2\},\ R_{i_{(v,m)}}\neq\emptyset. $ Indeed, if we suppose that there exists $ v^*\in \{k,k+1,\dots,\lambda-2\}$ such that $ R_{i_{(v^*,m)}}=\emptyset $, then from Claim \ref{cl_imp_3}(iv) it follows that $ R_{i_{(\lambda-1,m)}}=\emptyset $ where the last contradicts with the assumption of the (vi).\\
			 
			As we proved before, since $ R_{i_{(k,m)}}\neq\emptyset $ it follows that: $ t^{(m)}_{i_{(k+1,m)}}>t^{(m)}_{i_{(k,m)}}. $
			Similarly, since $ R_{i_{(k+1,m)}}\neq\emptyset $ it follows that: $ t^{(m)}_{i_{(k+2,m)}}>t^{(m)}_{i_{(k+1,m)}}. $
		    Continuing the above procedure for finite steps, we obtain finally
		    \begin{gather}\label{eq_5}
		    	t^{(m)}_{i_{(k,m)}}<t^{(m)}_{i_{(k+1,m)}}< t^{(m)}_{i_{(k+2,m)}}<\dots<t^{(m)}_{i_{(\lambda-1,m)}}
		    \end{gather}
		    Combining, $ (\ref{eq_4}) $ and $ (\ref{eq_5}) $ we conclude that $ t^{(m)}_{i_{(k,m)}}<t^{(m)}_{i_{(\lambda,m)}}. $
		    
			\item[(vii)] Let us suppose the opposite of what we seek to prove i.e $ \forall v\geq 2,\ R_{i_{(v-1,m)}}\neq\emptyset. $ From the Claim \ref{cl_imp_3}(iii), we obtain that the sequnece of indexes $ \left( i_{(v,m)}\right)_{v\geq 2}  $ consists of distinct terms. Therefore, the set $ \{i_{(v,m)}:\ v\in\N,\ v\geq 2\}\subset Q_m $ is infinite. Consequently, $ Q_m $ is infinite. The last contradicts with the Claim \ref{cl_imp_1}.
		\end{itemize}
	\end{proof}
From the above Claim $ \ref{cl_imp_3}(vii), $ we are in the position to define, for every $ m\in\{1,2,\dots,j_0\} $:
\begin{gather}
	v^{(m)}_0:=\min\{v\in\N:\ R_{i_{(v,m)}}=\emptyset\}.
\end{gather}
Therefore, we have the following strictly increasing\footnote{Look the Claim $ \ref{cl_imp_3}(vi) $} sequence $ \left( t^{(m)}_{i_{(v,m)}}\right)^{v^{(m)}_0}_{v=1} $ of points in $ [0,1] $ and the corresponding sequence $ \left( z^{(m)}\left( t^{(m)}_{i_{(v,m)}}\right) \right)^{v^{(m)}_0}_{v=1}  $ of points  in $ [x_{k_m-1},x_{k_m}]. $ From now on, we symbolize $ z^{(m)}_v:=z^{(m)}\left(t^{(m)}_{i_{(v,m)}} \right)  $ and $ z^{(m)}_{v,\ \max}:=z^{(m)}\left(t^{(m)}_{i_{(v,m),\ \max}} \right)  $.

			 \begin{figure}[h]
	\centering
	\includegraphics[width=160mm]{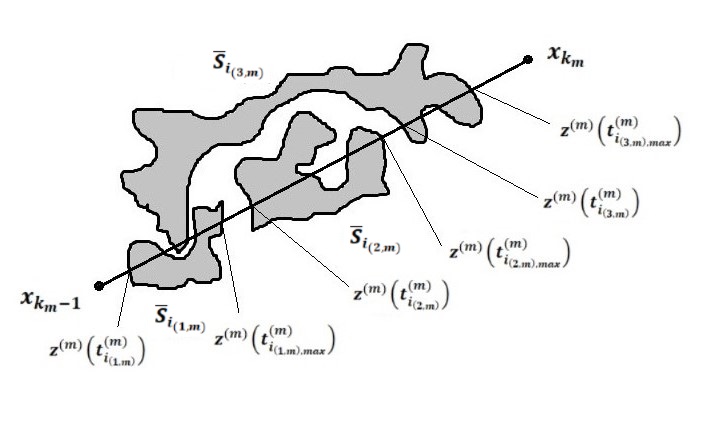}
	\caption{An example in $ \left(\R^d, \norm{\cdot}_2 \right)  $ that represents the points $  z^{(m)}\left( t^{(m)}_{i_{(v,m)}}\right),\ v\in\left\lbrace 1,2,\dots,v^{(m)}_0 \right\rbrace  $}
\end{figure}

Next, we proceed to the final part of the paper, the construction of the desired polygonal line in $ U\setminus \left(\bigcup_{i\in I}S_i \right)  $ which connects the points $ x\equiv x_1 $ and  $ y\equiv x_n.  $
\\
For every $ m\in\{1,2,\dots j_0\} $ there are two possible cases for $ x_{k_m}. $ In specific, $  x_{k_m}\in\bigcup_{i\in I}S_i $ or $ x_{k_m}\notin\bigcup_{i\in I}S_i. $ Thus, for each case for $ x_{k_m}, $ we construct a special type of polygonal line, namely $ \Gamma^{(1)}_{x_{k_m}}, $ if $ x_{k_m}\notin\bigcup_{i\in I}S_i $ and $ \Gamma^{(2)}_{x_{k_m}}, $ if $ x_{k_m}\in\bigcup_{i\in I}S_i $. Before proceeding to case separation, we construct a special type of polygonal line, $ \tilde{\Gamma}_{x_{k_m-1},\ x_{k_m}},\ \forall m\in\{1,2,\dots,j_0\} $, which as we will see later, will be a main part of the most complex ones polygonal lines $ \Gamma^{(1)}_{x_{k_m}},\ \Gamma^{(2)}_{x_{k_m}} $.

\begin{itemize}
	\item \textbf{Construction of Polygonal line} $ \tilde{\Gamma}_{x_{k_m-1},\ x_{k_m}},\  m\in\{1,2,\dots,j_0\} $\\
	In the following, for any $ \tilde{x},\tilde{y}\in K_{\delta_i}$  we denote by $ \hat{\Gamma}_{\tilde{x},\tilde{y}} $ any polygonal line in $ K_{\delta_i},\ i\in I $ (for a definition of $ K_{\delta_i}$, see theorem formulation) that connects the points  $ \tilde{x},\tilde{y}. $ This polygonal line exists by the assumption of the theorem, the set $ K_{\delta_i}\subset U\setminus\left(\bigcup_{i\in I}S_i \right)  $ is polygonally-connected.
	\begin{claim}\label{Polygonal_line_cl_1}
		For every $ m\in\{1,2,\dots,j_0\}, $ the following hold:
		\begin{itemize}
			\item[(i)] If $ x_{k_m-1}\notin\bigcup_{i\in I}S_i, $ then  $ [x_{k_m-1},\ z^{(m)}_1] \subset U\setminus \left( \bigcup_{i\in I}S_i\right) $
			\item[(ii)] If $  x_{k_m}\notin\bigcup_{i\in I}S_i, $ then  $  [z^{(m)}_{v_0,\ \max},\ x_{k_m}] \subset U\setminus \left( \bigcup_{i\in I}S_i\right) $
			\item[(iii)] If $ v_0^{(m)}\geq 3 $ then for every $ v\in\{1,2,\dots, v_0^{(m)}-1\},\ [z^{(m)}_{v,\ \max},\ z^{(m)}_{v+1}] \subset U\setminus \left( \bigcup_{i\in I}S_i\right)   $	
			\item[(iv)] Let $ x_{k_m-1}\in S_q $ for some $ q\in I. $ Then, $ q\in Q_m $ and also $ q=i_{(1,m)}. $ Respectively, if $ x_{k_m}\in S_q $ for some $ q\in I, $ then, $ q\in Q_m $
			\item[(v)] Let $ x_{k_m-1},x_{k_m}\in\cup_{i\in I}S_i.$ If $ v_0^{(m)}=1 $ then, there exists (unique) $ q\in I $ such that, $ x_{k_m-1}, x_{k_m}\in S_q. $
			\item[(vi)] If $ x_{k_m-1},\ x_{k_m}\in S_q $, for some $ q\in Q_m $,	 then $ v_0^{(m)}=1 $
			\item[(vii)] Let $ x_{k_m-1},x_{k_m}\in\cup_{i\in I}S_i.$ If $ v_0^{(m)}\geq 2 $ then, there exists (unique) $ q_1,q_2\in I $ with $ q_1\neq q_2, $ such that $ x_{k_m-1}\in S_{q_1}$ and $ x_{k_m}\in S_{q_2} $
			\item[(viii)] If $ x_{k_m-1}\in S_{q_1} $ and $ x_{k_m}\in S_{q_2} $ for some $ q_1,q_2\in Q_m $ with $ q_1\neq q_2, $ then $ v_0^{(m)}\geq 2 $
		\end{itemize}
	\end{claim}
\begin{proof}
	Let $ m\in\{1,2,\dots,j_0\}. $ 
	\begin{itemize}
		\item[(i)] Let $ x_{k_m-1}\notin\bigcup_{i\in I}S_i. $\\
		We prove that $ [x_{k_m-1},\ z^{(m)}_1] \subset U\setminus \left( \bigcup_{i\in I}S_i\right) $ i.e, $ \forall i\in I,\ [x_{k_m-1},\ z^{(m)}_1]\cap S_i=\emptyset.  $ Let us suppose the opposite of what se seek to prove: There exists $ i^*\in I,\ [x_{k_m-1},\ z^{(m)}_1] \cap S_{i^*}\neq\emptyset.$ Since,
		\begin{gather}
			[x_{k_m-1},\ z^{(m)}_1]:=\left\lbrace (1-t)x_{k_m-1}+tz^{(m)}_1:\ t\in[0,1] \right\rbrace=\left\lbrace (1-t)x_{k_m-1}+t x_{k_m}:\ t\in[0,t^{(m)}_{i_{(1,m)}}]\right\rbrace\nonumber  
		\end{gather}
		there exists $ t^*\in[0,t^{(m)}_{i_{(1,m)}}] $ such that $ z^*\equiv z^{(m)}(t^*):=(1-t^*)x_{k_m-1}+t^* x_{k_m}\in S_{i^*}. $ Initially we prove that $ t^*<t^{(m)}_{i_{(1,m)}}. $ For this task we need the following intermediate step:\\ 
		\underline{\textit{Intermediate Step (I)}}: $ i^*\neq i_{(1,m)}. $
		\begin{proof}
			Let $ i^*= i_{(1,m)} $. Thus, $ i^*\in Q_m. $ We notice that $ x_{k_m-1}\in\partial S_{i_{(1,m)}}\cup ext(S_{i_{(1,m)}}) $. In case where, $ x_{k_m-1}\in ext(S_{i_{(1,m)}}) $, since $ z^*\in S_{i_{(1,m)}} $  from the Proposition \ref{main_prop}, it follows that $ z^{(m)}([0,t^*]) \cap\partial S_{i_{(1,m)}}\neq\emptyset. $ The same holds, if we make the assumption that $ x_{k_m-1}\in\partial S_{i_{(1,m)}}. $ Therefore, there exists, $ \tilde{t}\in[0,t^*] $, such that $ \tilde{z}:=z^{(m)}(\tilde{t}) \in\partial S_{i_{(1,m)}}. $ We observe that, $ \tilde{t}<t^* $, since, if $ \tilde{t}=t^*  $ we extract that $ \tilde{z}=z^*, $ i.e $ \partial S_{i_{(1,m)}}\cap S_{i_{(1,m)}}\neq\emptyset, $  where the last is contradiction since $ S_{i_{(1,m)}} $ is open. Now, from the definition of $ L^{(m)}_{i_{(1,m)}}, $ we receive that, $ \tilde{t}\in L^{(m)}_{i_{(1,m)}}.  $ Moreover, from the definition of $ t^{(m)}_{i_{(1,m)}}:=\min\{t^{(m)}_{q,\ \min}:\ q\in Q_m\}, $ we receive that $ t^{(m)}_{i_{(1,m)}}\leq\min L^{(m)}_{i_{(1,m)}}\leq \tilde{t}.  $ We see that the last contradicts with the fact that, $ \tilde{t}<t^*\leq t^{(m)}_{i_{(1,m)}}. $ Consequently, $ i^*\neq\ i_{(1,m)}.$
			\end{proof}
			We observe that, $ t^*<t^{(m)}_{i_{(1,m)}}. $ Indeed, if $ t^*=t^{(m)}_{i_{(1,m)}} $, then $ z^*=z^{(m)}_1. $ Since, $ z^*\in S_{i^*} $ and $ z^{(m)}_1\in\partial S_{i_{(1,m)}} $ we obtain $ \bar{S}_{i^*}\cap \bar{S}_{i_{(1,m)}}\neq\emptyset, $ which is a contradiction, since from the Intermediate Step (I), $ i^*\neq\ i_{(1,m)} $ therefore by the assumption of theorem, $ \bar{S}_{i^*}\cap \bar{S}_{i_{(1,m)}}=\emptyset.  $\\
			\underline{\textit{Intermediate Step (II)}}: $ i^*\in Q_m $.
			\begin{proof}
				We work similarly as we did in the Intermediate Step (I). Initially, we know that, $ x_{k_m-1}\in\partial S_{i^*}\cap ext(S_{i^*}). $ If $  x_{k_m-1}\in\partial S_{i^*} $ then, since $ z^*\in S_{i^*} $, we conclude that $ i^*\in Q_m. $ Let now, $ x_{k_m-1}\in ext(S_{i^*}). $ Again, since $ z^*\in S_{i^*} $, from the Proposition \ref{main_prop} we receive, $ z^{(m)}([0,t^*])\cap\partial S_{i^*}\neq\emptyset. $ Therefore\footnote{It holds that $[x_{k_m-1},\ z^*]:=\left\lbrace (1-t)x_{k_m-1}+tz^*:\ t\in[0,1] \right\rbrace=\left\lbrace (1-t)x_{k_m-1}+t x_{k_m}:\ t\in[0,t^*]\right\rbrace    $}, $ [x_{k_m-1},\ z^*]\cap\partial S_{i^*}\neq\emptyset.(\text{and}\ [x_{k_m-1},\ z^*]\cap S_{i^*}\neq\emptyset,\ \text{since}\ z^*\in S_{i^*} ). $ Hence, $ i^*\in Q_m.$
			\end{proof}
			Let now $ \bar{z}\in [x_{k_m-1},\ z^*] \cap\partial S_{i^*}=z^{(m)}([0,t^*])\cap\partial S_{i^*}$, i.e there exists (unique) $ \bar{t}\in[0,\ t^*] $ such that $ \bar{z}=z^{(m)}(\bar{t}) $. By the definition of $ L^{(m)}_{i^*} $ we obtain, $ \bar{t}\in L^{(m)}_{i^*}  $. Furthermore, $ \bar{t}< t^* $. Indeed, if we suppose the equality, then we have $ \bar{z}= z^* $. Consequently, $ \partial S_{i^*}\cap S_{i^*}\neq \emptyset $ which is a contradiction since the set $ S_{i^*} $ is open. In addition, we know from the definition of $ t^{(m)}_{i_{(1,m)}} $, that $ t^{(m)}_{i_{(1,m)}}\leq \min L^{(m)}_{i^*}\leq \bar{t}. $ Therefore,  $ t^{(m)}_{i_{(1,m)}}\leq\bar{t}<t^* $. The last, contradicts with previous result, we proved, i.e $ t^*<t^{(m)}_{i_{(1,m)}}. $  From the above analysis, we finally receive the desired result, $ \forall i\in I,\ [x_{k_m-1},\ z^{(m)}_1]\cap S_i=\emptyset.  $
		\item[(ii)] Let $  x_{k_m}\notin\bigcup_{i\in I}S_i. $\\
		We prove that $ \left[ z^{(m)}_{v_0^{(m)},\ \max},\ x_{k_m}\right]  \subset U\setminus \left( \bigcup_{i\in I}S_i\right), $ i.e, $ \forall i\in I,\ \left[ z^{(m)}_{v_0^{(m)},\ \max},\ x_{k_m}\right] \cap S_i=\emptyset. $ Let us suppose the opposite of what we seek to prove. Therefore, there exists $ i^*\in I $, such that, $ \left[ z^{(m)}_{v_0^{(m)},\ \max},\ x_{k_m}\right] \cap S_{i^*}\neq\emptyset.  $
		From the last, we obtain that there exists $ z^*\in\left[ z^{(m)}_{v_0^{(m)},\ \max},\ x_{k_m}\right] \cap S_{i^*}. $ Thus, there exists\footnote{It holds that, $ \left[ z^{(m)}_{v_0^{(m)},\ \max},\ x_{k_m}\right] :=\left\lbrace(1-t) z^{(m)}_{v_0^{(m)},\ \max}+t x_{k_m}:\ t\in[0,\ 1] \right\rbrace=\left\lbrace (1-t) x_{k_m-1}+t x_{k_m}:\ t\in\left[ t^{(m)}_{i_{(v_0^{(m)},\ m)},\ \max},\ 1\right] \right\rbrace   $} $ t^*\in\left[t^{(m)}_{i_{\left( v_0^{(m)},\ m \right) },\ \max},\ 1\right]  $ such that $ z^*\equiv z^{(m)}(t^*)\in S_{i^*}. $ For convenience, we set $ t_{v_0^{(m)}}:=t^{(m)}_{i_{\left( v_0^{(m)},\ m \right) },\ \max}. $\\
		In order to lead to a contradiction, we prove the following Intermediate Steps.\\
		\underline{\textit{Intermediate Step (I)}}: $ t^*>t_{v_0^{(m)}} $
		\begin{proof}
			Let us suppose the opposite of what we seek to prove, i.e $ t^*=t_{v_0^{(m)}}. $ Then, $ z^*\equiv z^{(m)}(t^*)=z^{(m)}\left( t_{v_0^{(m)}}\right) \equiv z^{(m)}_{v_0^{(m)},\ \max} $. Since, $ z^*\in S_{i^*} $ and $ z^{(m)}_{v_0^{(m)},\ \max}\in\partial S_{i_{\left( v_0^{(m)},\ m\right) }} $ we conclude that $ S_{i^*}\cap\partial S_{i_{\left( v_0^{(m)},\ m\right) }}\neq\emptyset.  $ If $ i^*=i_{(v_0,\ m)} $ then the expression $ S_{i^*}\cap\partial S_{i_{\left( v_0^{(m)},\ m\right) }}\neq\emptyset  $ leads to a contradiction, since the set $ S_{i^*}=S_{i_{\left( v_0^{(m)},\ m\right) }} $ is open. On the other hand, if $ i^*\neq i_{\left( v_0^{(m)},\ m\right) } $ then using the fact that $ S_{i^*}\cap\partial S_{i_{\left( v_0^{(m)},\ m\right) }}\neq\emptyset $ we conclude $ \bar{S}_{i^*}\cap\bar{S}_{i_{\left( v_0^{(m)},\ m\right) }}\neq\emptyset, $ which is a contradiction from the assumption of the theorem. Consequently, we obtain $ t^*>t_{v_0^{(m)}}. $ 
		\end{proof}
		\underline{\textit{Intermediate Step (II)}}: $ i^*\in Q_m $ and in specific, $ [z^*,\ x_{k_m}]\cap\partial S_{i^*}\neq\emptyset $.
		\begin{proof}
			Initially, we observe that $ [z^*,x_{k_m}]\cap S_{i^*}\neq\emptyset, $ therefore $ [x_{k_m-1},x_{k_m}]\cap S_{i^*}\neq\emptyset. $ By the assumption we have that $ x_{k_m}\notin S_{i^*} $. Thus, $ x_{k_m}\in \partial S_{i^*}\cup ext(S_{i^*}). $ We distinguish the following cases:
			\begin{itemize}
				\item Let $ x_{k_m}\in\partial S_{i^*}. $ Then $ [x_{k_m-1},x_{k_m}]\cap \partial S_{i^*}\neq\emptyset,$ therefore $ i^*\in Q_m. $
				\item Let $ x_{k_m}\in ext(S_{i^*}). $ Since $ z^*\in S_{i^*}, $ from Proposition \ref{main_prop}, it holds that 
				\begin{gather}
					z^{(m)}([t^*,1])\cap\partial S_{i^*}\neq\emptyset
				\end{gather}
			Furthermore, since $ [z^*,\ x_{k_m}]=z^{(m)}([t^*,\ 1])\subset z^{(m)}([0, 1])=[x_{k_m-1},\ x_{k_m}] $, we obtain $ [x_{k_m-1},\ x_{k_m}] \cap\partial S_{i^*}\neq\emptyset, $ therefore $ i^*\in Q_m. $
			\end{itemize}
		From both cases, we also obtain, $ [z^*,\ x_{k_m}]\cap\partial S_{i^*}\neq\emptyset. $
		\end{proof}
		Let $ \tilde{z}\in[z^*,x_{k_m}]\cap\partial S_{i^*}. $ I.e. there exists $ \tilde{t}\in[t^*,1],\ \tilde{z}=z^{(m)}(\tilde{t}). $ We observe that, $ \tilde{t}>t^*. $ Indeed, if $ \tilde{t}=t^* $ then $ \tilde{z}=z^*. $ From the last, we receive that $ S_{i^*}\cap\partial S_{i^*}\neq\emptyset $ which contradicts with the fact that $ S_{i^*} $ is open set. Therefore, $ \tilde{t}>t^*. $ Using the Intermediate Step (I), we have $ \tilde{t}>t^*>t_{v_0^{(m)}}. $\\
		\underline{\textit{Intermediate Step (III)}}: $ i^*\in Q_m \setminus\left\lbrace i_{(1,m)},\ i_{(2,m)},\dots, i_{\left( v_0^{(m)},\ m\right) }\right\rbrace.  $
		\begin{proof}
			From the Intermediate Step (II), we already know that $ i^*\in Q_m. $ We now prove  that, $ i^*\neq i_{(\lambda,m)},\ \forall\lambda\in\left\lbrace 1,2,\dots,v_0^{(m)}\right\rbrace.  $ First we see that, 
			\begin{gather}
				\tilde{t}>t^*>t_{v_0^{(m)}}\equiv t^{(m)}_{i_{\left( v_0^{(m)},\ m\right) },\ \max}\geq t^{(m)}_{i_{\left(v_0^{(m)},\ m \right)} }>t^{(m)}_{i_{\left( v_0^{(m)}-1,\ m\right) },\ \max}\geq t^{(m)}_{i_{\left(v_0^{(m)}-1,\ m\right) }}>\dots>t^{(m)}_{i_{(1,\ m)},\ \max}
			\end{gather}
		Let us suppose that there exists $ \tilde{\lambda}\in\left\lbrace 1,2,\dots,v_0^{(m)} \right\rbrace$ such that $ i^*=i_{(\tilde{\lambda},\ m)} $. Thus, $ L^{(m)}_{i^*}=L^{(m)}_{i_{(\tilde{\lambda},\ m)}} $. At the same time, we know that  $ \tilde{t}\in L^{(m)}_{i^*}, $ since $ \tilde{z}=z^{(m)}(\tilde{t})\in \partial S_{i^*}. $ Hence, $ \tilde{t}\in L^{(m)}_{i_{(\tilde{\lambda},\ m)}}. $ By the definition of $ t^{(m)}_{i_{(\tilde{\lambda},\ m)},\ \max}, $ we obtain that  $ \tilde{t}\leq t^{(m)}_{i_{(\tilde{\lambda},\ m)},\ \max}, $ which contradicts with the fact that  $\tilde{t}>t^{(m)}_{i_{(\tilde{\lambda},\ m)},\ \max}.  $ Consequently, $ i^*\neq i_{(\lambda,\ m)},\ \forall\lambda\in\left\lbrace 1,2,\dots, v_0^{(m)} \right\rbrace.  $
		\end{proof}
	   \underline{\textit{Intermediate Step (IV)}}: $ \left[ z^{(m)}_{v_0^{(m)},\ \max},\ z^*\right] \cap\partial S_{i^*}\neq\emptyset$.
		\begin{proof}
			We know that $ z^{(m)}_{v_0^{(m)},\ \max}\in\partial S_{i_{\left(v_0^{(m)},\ m \right) }}\subset ext (S_{i^*}) $ and $ z^*\in S_{i^*}. $ Thus, from the Proposition \ref{main_prop},we obtain that $ z^{(m)}\left(\left[t_{v_0^{(m)}},\ t^* \right]  \right)\cap\partial S_{i^*}\neq\emptyset.  $
		\end{proof}
	Let $ z^{**}\in \left[ z^{(m)}_{v_0^{(m)},\ \max},\ z^*\right] \cap\partial S_{i^*},  $ i.e. there exists\footnote{It holds that $ \left[ z^{(m)}_{v_0^{(m)},\ \max},\ z^*\right]:=\left\lbrace (1-t)\ z^{(m)}_{v_0^{(m)},\ \max}+ t\ z^*:\ t\in[0,\ 1] \right\rbrace=\left\lbrace (1-t)\ x_{k_m-1}+ t\ x_{k_m}:\ t\in\left[ t_{v_0^{(m)}},\ t^*\right]  \right\rbrace  $} $ t^{**}\in \left[t_{v_0^{(m)}},\ t^* \right]  $ such that $ z^{**}=z^{(m)}(t^{**}). $
	\\
	   \underline{\textit{Intermediate Step (V)}}: $ t^*>t^{**}>t_{v_0^{(m)}} $.
		\begin{proof}
		If $ t^{**}=t_{v_0^{(m)}}, $ then $ z^{**}= z^{(m)}_{v_0^{(m)},\ \max} $  where $ z^{**}\in\partial S_{i^*} $ and $  z^{(m)}_{v_0^{(m)},\ \max}\in\partial S_{i_{\left(v_0^{(m)},\ m \right) }}. $ Hence, $ \partial S_{i^*}\cap\partial S_{i_{\left(v_0^{(m)},\ m \right) }}\neq\emptyset. $ Therefore, $ \bar{S}_{i^*}\cap \bar{S}_{i_{\left(v_0^{(m)},\ m \right) }}\neq\emptyset.  $ The last leads to a contradiction, since, from the Intermediate Step (III), we have $ i^*\neq i_{\left(v_0^{(m)},\ m \right) } $ and as a result from the assumption of the theorem, we receive that $  \bar{S}_{i^*}\cap \bar{S}_{i_{\left(v_0^{(m)},\ m \right) }}=\emptyset.$ Consequently, $ t^{**}>t_{v_0^{(m)}}. $\\
		On the other hand, if $ t^*=t^{**}, $ it follows that $ z^*=z^{**}. $ Therefore, $ S_{i^*}\cap\partial S_{i^*}\neq\emptyset. $ The last contradicts with the fact that $ S_{i^*} $ is open set.
		\end{proof}
		\underline{\textit{Intermediate Step (VI)}}: $ t^{**}\in R_{i_{\left( v_0^{(m)},\ m\right) }} $.
		 	\begin{proof}
		 	Initially, we observe that $ t^{**}\in \bigcup \left\lbrace L^{(m)}_q:\ q\in Q_m\setminus\left\lbrace i_{(1,m)},\ i_{(2,m)},\dots,i_{\left( v_0^{(m)},\ m\right) } \right\rbrace\right\rbrace   $. Indeed, since $ z^{**}=z^{(m)}(t^{**})\in\partial S_{i^*},   $ it follows from the definition of $ L^{(m)}_{i^*} $ that $ t^{**}\in L^{(m)}_{i^*}. $ Moreover, from the Intermediate Step (III) we know that $ i^*\in Q_m\setminus\left\lbrace i_{(1,m)},\ i_{(2,m)},\dots,i_{\left(v_0^{(m)},\ m \right) } \right\rbrace.  $ Consequently, we receive that $ t^{**}\in \bigcup \left\lbrace L^{(m)}_q:\ q\in Q_m\setminus\left\lbrace i_{(1,m)},\ i_{(2,m)},\dots,i_{\left( v_0^{(m)},\ m\right)} \right\rbrace\right\rbrace. $ Furthermore, from the Intermediate Step (V), we know that $ t^{**}>t_{v_0^{(m)}}. $ It remains to prove the existence of $ \delta_{t^{**}}>0 $ such that $ z^{(m)}(t^{**}+\delta_{t^{**}})\in\bigcup \left\lbrace S_q:\ q\in Q_m\setminus\left\lbrace i_{(1,m)},\ i_{(2,m)},\dots,i_{\left( v_0^{(m)},\ m\right) } \right\rbrace\right\rbrace.$ Indeed, we set $ \delta_{t^{**}}:=t^*-t^{**}>0 $, where $\delta_{t^{**}}>0,  $ since from Intermediate Step (V) we know that $ t^*>t^{**}. $\\ Moreover, $ z^{(m)}(t^{**}+\delta_{t^{**}})=z^{(m)}(t^*)=z^*\in S_{i^*}\subset\bigcup \left\lbrace S_q:\ q\in Q_m\setminus\left\lbrace i_{(1,m)},\ i_{(2,m)},\dots,i_{\left( v_0^{(m)},\ m\right) } \right\rbrace\right\rbrace. $ Hence, $ t^{**}\in R_{i_{\left( v_0^{(m)},\ m\right) }}. $
		 \end{proof}
		From the last Intermediate Step (VI), we are in a contradiction, since $ R_{i_{\left( v_0^{(m)},\ m\right) }}=\emptyset $ (see above the definition of $ v_0^{(m)}. $ ) Finally, we receive the desired result, $  \left[ z^{(m)}_{v_0^{(m)},\ \max},\ x_{k_m}\right]  \subset U\setminus \left( \bigcup_{i\in I}S_i\right). $ 
		
		\item[(iii)] Let $ v_0^{(m)}\geq 3 $ and a random $ v^*\in\left\lbrace 1,2,\dots, v_0^{(m)}-1\right\rbrace  $. We prove that, $ \left[ z^{(m)}_{v^*,\ \max},\ z^{(m)}_{v^*+1}\right] \subset U\setminus \left( \bigcup_{i\in I}S_i\right). $ Let us suppose the opposite of what we seek to prove. Therefore, there exists $ i^*\in I, $ such that $ \left[z^{(m)}_{v,\max},\ z^{(m)}_{v+1} \right]\cap S_{i^*}\neq\emptyset.  $ From the last, we obtain that there exists $ z^*\in \left[z^{(m)}_{v,\max},\ z^{(m)}_{v+1} \right]\cap S_{i^*}. $ Thus, there exists\footnote{It holds that $ \left[z^{(m)}_{v,\max},\ z^{(m)}_{v+1} \right]:=\left\lbrace (1-t)z^{(m)}_{v,\max}+tz^{(m)}_{v+1}:\ t\in[0,1]\right\rbrace =\left\lbrace (1-t)x_{k_m-1}+t x_{k_m}:\ t\in\left[ t^{(m)}_{i_{(v,m)},\max},\ t^{(m)}_{i_{(v+1,m)}}\right] \right\rbrace  $} $ t^*\in \left[t^{(m)}_{i_{(v,m)},\max},t^{(m)}_{i_{(v+1,m)}} \right]  $ such that $ z^*\equiv z^{(m)}(t^*)\in S_{i^*}. $ In order to lead to a contradiction, we make the following Intermediate Steps, where their proofs are omitted since is similar with the proofs we did in (ii).\\
		\underline{\textit{Intermediate Step (I)}}: $ t^*< t^{(m)}_{i_{(v+1,m)}}.  $\\
		\underline{\textit{Intermediate Step (II)}}: $ i^*\in Q_m $ and in specific $ \left[ z^*,\ z^{(m)}_{v+1} \right]\cap\partial S_{i^*}\neq\emptyset  $.\\
		Let $ \tilde{z}\in\left[ z^*,\ z^{(m)}_{v+1} \right]\cap\partial S_{i^*}, $ i. e. there exists\footnote{It holds $ \left[ z^*, z^{(m)}_{v+1}\right]:=\left\lbrace (1-t) z^* +t z^{(m)}_{v+1}:\ t\in[0,1] \right\rbrace=\left\lbrace (1-t)x_{k_m-1}+t x_{k_m}:\ t\in\left[t^*,t^{(m)}_{i_{(v+1,m)}} \right] \right\rbrace    $} $ \tilde{t}\in\left[t^*,\ t^{(m)}_{i_{(v+1,m)}} \right]  $ such that $ \tilde{z}=z^{(m)}(\tilde{t}) $ We observe that, $ \tilde{t}>t^*. $ Indeed, if $ \tilde{t}=t^* $ then $ \tilde{z}=z^*. $ From the last, we receive that $ S_{i^*}\cap\partial S_{i^*}\neq\emptyset $ which contradicts with the fact that $ S_{i^*} $ is open set. Therefore, $ \tilde{t}>t^*. $ Using now the Intermediate Step (I), we receive $ \tilde{t}>t^*>t^{(m)}_{i_{(v,m)},\ \max}. $ \\
		\underline{\textit{Intermediate Step (III)}}: $ i^*\in Q_m\setminus\left\lbrace i_{(1,m)},\ i_{(2,m)},\dots,i_{(v,m)} \right\rbrace  $.\\
		\underline{\textit{Intermediate Step (IV)}}: $ \left[z^{(m)}_{v,\ \max},z^* \right]\cap\partial S_{i^*}\neq\emptyset  $.\\
		Let $ z^{**}\in\left[z^{(m)}_{v,\ \max},z^* \right]\cap\partial S_{i^*}, $ i. e. there exists $ t^{**}\in\left[t^{(m)}_{i_{(v,m)},\ \max},t^* \right]  $ such that $ z^{**}=z^{(m)}(t^{**}). $\\
		\underline{\textit{Intermediate Step (V)}}: $ t^*>t^{**}>t^{(m)}_{i_{(v,m)},\ \max}$\\
		Therefore, from the Intermediate Steps (I) and (V), we receive that $ t^{(m)}_{i_{(v,m)},\ \max}<t^{**}<t^{(m)}_{i_{(v+1,m)}}. $\\
		\underline{\textit{Intermediate Step (VI)}}: $ t^{**}\in R_{i_{(v,m)}} $\\
		At this stage we have a contradiction,  since by the definition of $ t^{(m)}_{i_{(v+1,m)}}=\min R_{i_{(v,m)}} $ but $ t^{**}\in R_{i_{(v,m)}} $ and $ t^{**}<t^{(m)}_{i_{(v+1,m)}}. $	Finally, we receive that $ \left[ z^{(m)}_{v^*,\ \max},\ z^{(m)}_{v^*+1}\right] \subset U\setminus\left(\bigcup_{i\in I}S_i\right). $

		\item[(iv)] Let $ x_{k_m-1}\in S_q $ for some $ q\in I. $ We prove first that $ q\in Q_m $ and in the end, that $ q=i_{(1,m)}. $\\
		Since, $ x_{k_m-1}\in S_q, $ it follows that, $ [x_{k_m-1},\ x_{k_m}]\cap S_q\neq\emptyset. $ Let us suppose that $ q\notin Q_m. $ Thus, $ [x_{k_m-1},\ x_{k_m}]\cap \partial S_q=\emptyset. $ From previous Claim, we know that $ Q_m\neq\emptyset. $ Let $ \tilde{q}\in Q_m. $ Definitely, $ q\neq\tilde{q}. $ Hence, $ \bar{S}_q\cap\bar{S}_{\tilde{q}}=\emptyset. $ Moreover, $ [x_{k_m-1},\ x_{k_m}]\cap\partial S_{\tilde{q}}\neq\emptyset. $ Let, $ \tilde{z}\in [x_{k_m-1},\ x_{k_m}]\cap\partial S_{\tilde{q}}, $ i. e. there exists, $ \tilde{t}\in[0,1] $ such that, $ \tilde{z}=z^{(m)}(\tilde{t})\in\partial S_{\tilde{q}}\subset ext(S_q).$ Using Proposition \ref{main_prop}, we receive that, $ z^{(m)}\left([0,\tilde{t}] \right)\cap\partial S_q\neq\emptyset.  $ Thus, $ \emptyset\neq[x_{k_m-1},\ \tilde{z}]\cap\partial S_q\subset[x_{k_m-1},\ x_{k_m}]\cap \partial S_q. $ The last contradicts with the previous statement, $ [x_{k_m-1},\ x_{k_m}]\cap \partial S_q=\emptyset. $ From the above, analysis, we conclude that, $ q\in Q_m. $\\
		Now, we proceed to the proof that, $ q=i_{(1,m)}. $ Let us suppose the opposite of what we seek to prove, i. e $ q\neq i_{(1,m)}. $ Therefore, by the assumption of the theorem, $ \bar{S}_{q}\cap \bar{S}_{i_{(1,m)}}=\emptyset.$ Since, $ x_{k_m-1}\in S_q $ and $ z^{(m)}_1\in\partial S_{i_{(1,m)}}, $ (i.e. $ z^{(m)}_1\in ext(S_q) $) it follows from the Proposition \ref{main_prop}, that $ z^{(m)}\left( [0,t^{(m)}_{i_{(1,m)}}]\right) \cap\partial S_{q}\neq\emptyset. $ Let $ z^*=z^{(m)}(t^*)\in\partial S_q $ for $ t^*\in\left[ 0, t^{(m)}_{i_{(1,m)}}\right].  $ It holds that, $ t^*<t^{(m)}_{i_{(1,m)}}.  $ Indeed, if we suppose that, $ t^*=t^{(m)}_{i_{(1,m)}} $, then it follows that $ \partial S_q\ni z^*=z^{(m)}\in\partial S_{i_{(1,m)}}. $ Thus, $ \bar{S}_{q}\cap \bar{S}_{i_{(1,m)}}\neq\emptyset$ which is a contradiction. Consequently, $ t^*<t^{(m)}_{i_{(1,m)}}. $ Furthermore, since $ t^*\in L^{(m)}_q, $ by the definition of, $ t^{(m)}_{q,\ \min} $, we obtain that $ t^{(m)}_{q,\ \min} \leq t^*<t^{(m)}_{i_{(1,m)}}. $ On the other hand, by the defintion of $ t^{(m)}_{i_{(1,m)}}=\min\left\lbrace t^{(m)}_{i_{(s,\ \min)}}:\ s\in Q_m \right\rbrace,  $ we receive that $ t^{(m)}_{i_{(1,m)}}\leq t^{(m)}_{q,\ \min} $ where we are in contradiction with the previous finding. \\
		With a similar, way we prove the case, $ x_{k_m}\in S_q. $
		\item[(v)]	Let $ v_0^{(m)}=1 $ and let $ x_{k_m-1},\ x_{k_m}\in\bigcup_{i\in I}S_i.  $ Therefore, there exists, (unique) $ q_1,q_2\in I $ such that $ x_{k_m-1}\in S_{q_1} $ and $ x_{k_m}\in S_{q_2} $. From the Claim \ref{Polygonal_line_cl_1} (iv), we receive that $ q_1,q_2\in Q_m $ and also $ q_1=i_{(1,m)}. $ Let us suppose that $ q_1\neq q_2. $ Since, $ z^{(m)}_{1,\ \max}\in\partial S_{i_{(1,m)}}\subset ext(S_{q_2}) $ and $ x_{k_m}\in S_{q_2}, $ from the Proposition \ref{main_prop} we receive that
		\begin{gather}
			z^{(m)}\left(\left[t^{(m)}_{i_{(1,m)},\ \max},\ 1 \right]  \right) \cap\partial S_{q_2}\neq\emptyset.
		\end{gather}
	  Let $ z^*=z^{(m)}(t^*)\in\partial S_{q_2}, $ where $ t^*\in \left[t^{(m)}_{i_{(1,m)},\ \max},\ 1 \right]  $. We observe that $ t^{(m)}_{i_{(1,m)}}<t^*<1. $ Indeed, $ t^*<1, $ since if we suppose that $ t^*=1 $ we obtain that $ z^*=z^{(m)}(t^*)= z^{(m)}(1)=x_{k_m}$ where $ z^*\in\partial S_{q_2} $ and $ z^{(m)} (1)=x_{k_m}\in S_{q_2}.$ Therefore, $ S_{q_2}\cap \partial S_{q_2}\neq\emptyset, $ which contradicts with the fact $ S_{q_2} $ is an open set. Next, we claim that $ t^{(m)}_{i_{(1,m)}}<t^*. $ This is true from the fact that if we suppose that $ t^{(m)}_{i_{(1,m)}}=t^*, $ we obtain $ \partial S_{q_1}\ni z^{(m)}_{1,\ \max}=z^*\in\partial S_{q_2},  $ thus $ \bar{S}_{q_1}\cap\bar{S}_{q_2}\neq\emptyset. $ On the other hand, from the condition $q_1\neq q_2  $ it follows from assumption of theorem,that $ \bar{S}_{q_1}\cap\bar{S}_{q_2}=\emptyset. $ Consequently, $ t^{(m)}_{i_{(1,m)}}<t^*. $
	  Furthermore, $ t^*\in R_{i_{(1,m)}}. $ Indeed, first we notice that $ q_2\in Q_m\setminus\{q_1=i_{(1,m)}\} $. Also, by the definition of $ t^*, $ it is obvious that $ t^*\in L^{(m)}_{q_2}. $ In addition, if we set as $ \delta_{t^*}:=1-t^*>0 $ (since $ t^*<1 $), then we obtain, that $ z^{(m)}(t^*+\delta_{t^*})=z^{(m)}(1)=x_{k_m}\in S_{q_2}. $ Therefore, by the definition of $  R_{i_{(1,m)}}, $ we receive that $ t^*\in R_{i_{(1,m)}}. $ From the last, we also conclude that, $  R_{i_{(1,m)}}\neq\emptyset. $ Therefore, $ v_0^{(m)}\geq 2, $ which contradicts with the assumption of Claim \ref{Polygonal_line_cl_1} (v), that $ v_0^{(m)}=1. $
	  
	  \item [(vi)] Let $ x_{k_m-1},x_{k_m}\in S_{q}, $ for some $ q\in Q_m. $ From the Claim \ref{Polygonal_line_cl_1} (iv), it holds that $ q=i_{(1,m)}. $\\
		\underline{\textit{Intermediate Step (I)}}: $ \forall t>t^{(m)}_{i_{(1,m)},\ \max},\ z^{(m)}(t)\in S_{i_{(1,m)}} $.
		\begin{proof}
			Let us suppose the opposite of what we seek to prove, i. e. there exists $ \tilde{t}>t^{(m)}_{i_{(1,m)}} $ such that $ \tilde{z}\equiv z^{(m)}(\tilde{t})\neq S_{i_{(1,m)}}. $ Then, from \ref{main_prop}, it holds that $ z^{(m)}\left([\tilde{t},\ 1] \right)\cap\partial S_{i_{(1,m)}}\neq\emptyset.  $ Let $ z^*=z^{(m)}(t^*)\in \partial S_{i_{(1,m)}} $ where $ t^*\in[\tilde{t},\ 1]. $ Thus, $ t^*\geq\tilde{t}>t^{(m)}_{i_{(1,m)},\ \max}. $ At the same time, by the definition of $ L^{(m)}_{i_{(1,m)}} $, it holds that $ t^*\in L^{(m)}_{i_{(1,m)}}. $ Therefore, $ t^*\leq t^{(m)}_{i_{(1,m)},\ \max}=\sup L^{(m)}_{i_{(1,m)}}.  $ The last contradicts with the previous finding, $t^*>t^{(m)}_{i_{(1,m)}}.  $ Consequently, the proof is complete.
		\end{proof}
		Next, we claim that $ R_{i_{(1,m)}}=\emptyset. $ Indeed, if $ R_{i_{(1,m)}}\neq\emptyset $ then there exists $ t'\in L^{(m)}_{q^*} $ where $ q^*\in Q_m\setminus\{i_{(1,m)}\} $ such that $ t'>t^{(m)}_{i_{(1,m)},\ \max}, $ i. e. $ z^{(m)}(t')\in\partial S_{q^*}\subset S_{i_{(1,m)}}. $ The last contradicts with the Intermediate Step (I). Consequently, $ R_{i_{(1,m)}}=\emptyset  $, hence $ v_0^{(m)}=1. $

		\item[(vii)] Let $ v_0^{(m)}\geq 2 $ and $ x_{k_m-1}, x_{k_m}\in\bigcup_{i\in I} S_i. $ Thus, there exist $ q_1,q_2\in I, $ such that $ x_{k_m-1}\in S_{q_1} $ and $ x_{k_m}\in S_{q_2}. $ From the Claim \ref{Polygonal_line_cl_1} (iv) it holds that $ q_1,q_2\in Q_m. $ Let us suppose that $ q_1=q_2=q $, i. e. $ x_{k_m-1},x_{k_m}\in S_q. $ Using now the Claim \ref{Polygonal_line_cl_1}(vi), we obtain that $ v_0^{(m)}=1. $ The last contradicts with the assumption of the Claim \ref{Polygonal_line_cl_1} (vii), $ v_0^{(m)}\geq 2. $ Hence, $ q_1\neq q_2. $
		\item[(viii)] Let $ x_{k_m-1}\in S_{q_1} $ and $ x_{k_m}\in S_{q_2}, $ where $ q_1,q_2\in Q_m $ with $ q_1\neq q_2. $ Let us suppose the opposite of what we seek to prove, i. e $ v_0^{(m)}=1. $ Then, from the Claim \ref{Polygonal_line_cl_1} (v) there exists $ q\in I $ such that $ x_{k_m-1},x_{k_m}\in S_q. $ Therefore, $ S_{q_1}\cap S_q\neq\emptyset $ and $ S_{q_2}\cap S_q\neq\emptyset $. If $ q_1\neq q $ then by the assumption of the theorem, we receive that $ S_{q_1}\cap S_q=\emptyset,  $ which contradicts with the previous finding. Consequently, $ q_1=q. $ Similarly, we obtain that $ q_2=q. $ Thus, $ q_1=q_2, $ which leads to a contradiction with the assumpion of Claim \ref{Polygonal_line_cl_1} (viii). Hence, $ v_0^{(m)}\geq 2. $

	\end{itemize}

\end{proof}
	We distinguish the following subcases:\\
	\textbf{Case:} $ x_{k_m-1},\ x_{k_m}\notin\bigcup_{i\in I}S_i $
	\\
	\begin{gather}
			\tilde{\Gamma}_{x_{k_m-1},\ x_{k_m}}:=\begin{cases}
				[x_{k_m-1},z^{(m)}_1]\cup\hat{\Gamma}_{z^{(m)}_1,\ z^{(m)}_{1,\ \max}}\cup[z^{(m)}_{1,\ \max},\ x_{k_m}],\ & v_0^{(m)}=1\nonumber\\
				[x_{k_m-1},\ z^{(m)}_1]\cup\left(\bigcup_{v=1}^{v_0^{(m)}-1}\left(  \hat{\Gamma}_{z^{(m)}_v,\ z^{(m)}_{v,\ \max}}\cup[z^{(m)}_{v,\ \max},\ z^{(m)}_{v+1}]\right) \right) \cup\\ \cup\left(\hat{\Gamma}_{z^{(m)}_{v_0^{(m)}},\ z^{(m)}_{v_0^{(m)},\ \max}}\cup[z^{(m)}_{v_0^{(m)},\ \max},\ x_{k_m}] \right),\ & v_0^{(m)}\geq 2 \nonumber
			\end{cases}
		\end{gather}
	\begin{figure}[htp]
		\centering
		\subfloat[Polygonal Line $  \tilde{\Gamma}_{x_{k_m-1},\ x_{k_m}}  $ (with red color) where $ x_{k_m-1},\ x_{k_m}\notin\bigcup_{i\in I}S_i $, with $ v^{(m)}_0=1. $ ]{%
			\includegraphics[clip,width=0.76\columnwidth]{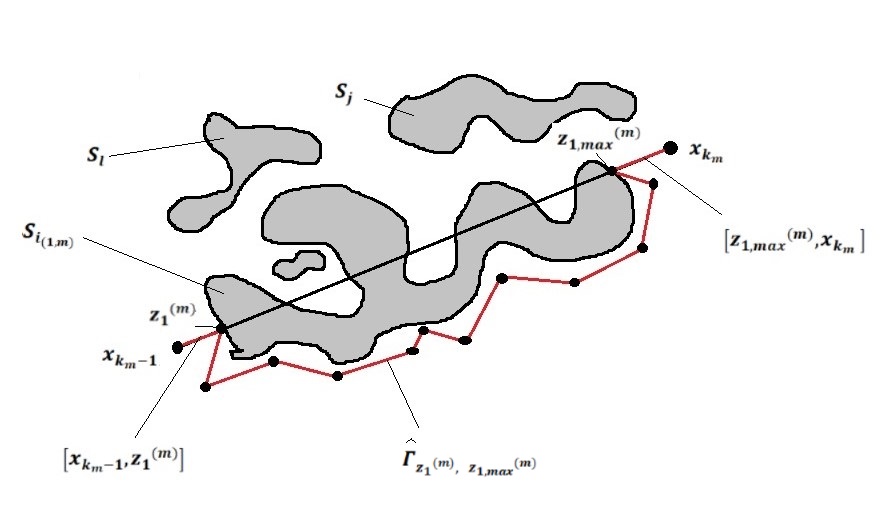}%
		}
		
		\subfloat[Polygonal Line $  \tilde{\Gamma}_{x_{k_m-1},\ x_{k_m}}  $ (with red color) where $ x_{k_m-1},\ x_{k_m}\notin\bigcup_{i\in I}S_i $, with $ v^{(m)}_0\geq 2. $]{%
			\includegraphics[clip,width=0.76\columnwidth]{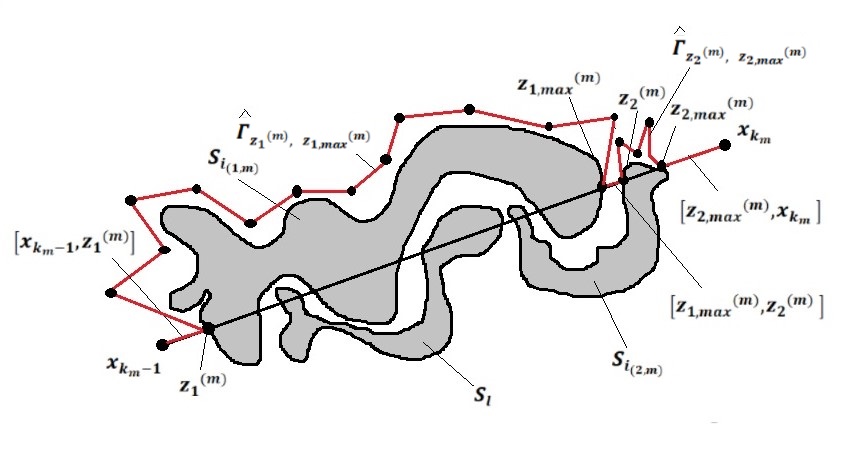}%
		}
		
		\caption{An example in $ \left(\R^d, \norm{\cdot}_2 \right)  $ that represents the Polygonal line  $  \tilde{\Gamma}_{x_{k_m-1},\ x_{k_m}}  $ (with red color) in case where $ x_{k_m-1},\ x_{k_m}\notin\bigcup_{i\in I}S_i $.}
	\end{figure}
	\newpage

	\begin{figure}[htp]
		\centering
		\subfloat[Polygonal Line $  \tilde{\Gamma}_{x_{k_m-1},\ x_{k_m}}  $ (with red color) where $ x_{k_m-1}\in \bigcup_{i\in I}S_i\  \text{and}\  x_{k_m}\notin\bigcup_{i\in I}S_i $, with $ v^{(m)}_0=1. $]{%
			\includegraphics[clip,width=0.85\columnwidth]{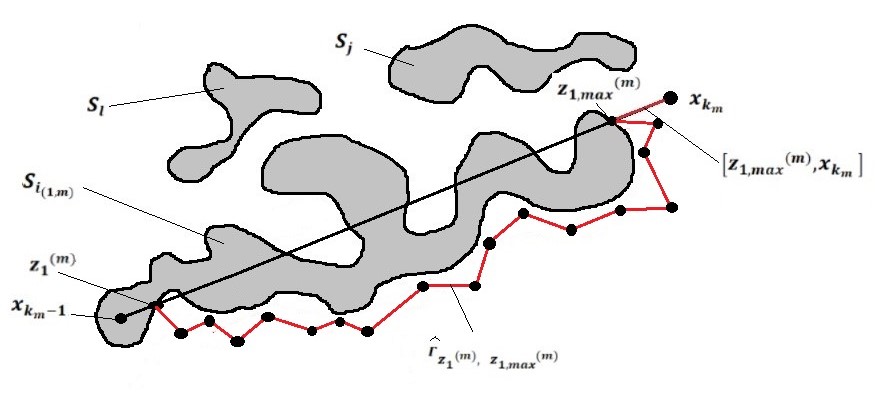}%
		}
		
		\subfloat[Polygonal Line $  \tilde{\Gamma}_{x_{k_m-1},\ x_{k_m}}  $ (with red color) where $ x_{k_m-1}\in \bigcup_{i\in I}S_i\  \text{and}\  x_{k_m}\notin\bigcup_{i\in I}S_i $, with $ v^{(m)}_0\geq 2. $]{%
			\includegraphics[clip,width=0.85\columnwidth]{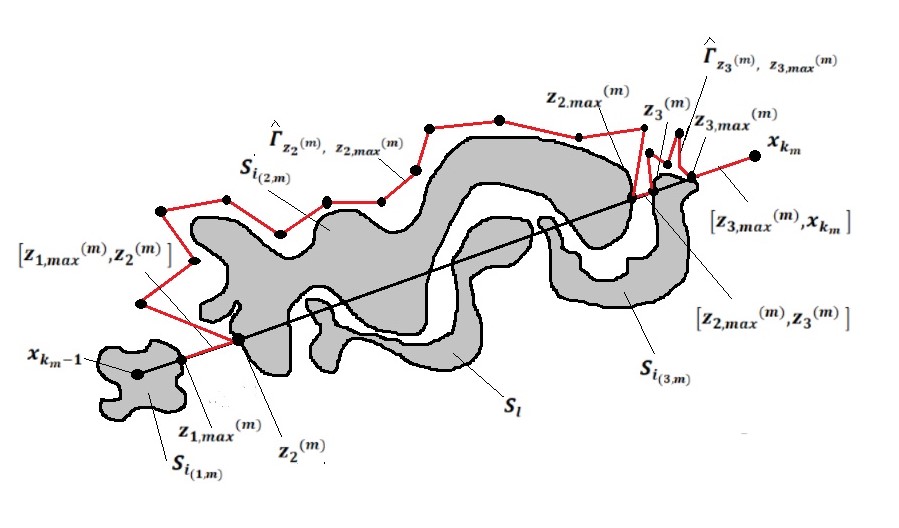}%
		}
		
		\caption{An example in $ \left(\R^d, \norm{\cdot}_2 \right)  $ that represents the Polygonal line  $  \tilde{\Gamma}_{x_{k_m-1},\ x_{k_m}}  $ (with red color) in case where $ x_{k_m-1}\in \bigcup_{i\in I}S_i\  \text{and}\  x_{k_m}\notin\bigcup_{i\in I}S_i $.}
	\end{figure}
	
		\textbf{Case:} $ x_{k_m-1}\in \bigcup_{i\in I}S_i\  \text{and}\  x_{k_m}\notin\bigcup_{i\in I}S_i $
	\\
	\begin{gather}
			\tilde{\Gamma}_{x_{k_m-1},\ x_{k_m}}:=\begin{cases}
				\hat{\Gamma}_{z^{(m)}_1,\ z^{(m)}_{1,\ \max}}\cup[z^{(m)}_{1,\ \max},x_{k_m}],\ & v_0^{(m)}=1\nonumber\\
				\left(\bigcup_{v=1}^{v_0^{(m)}-1}\left(  \hat{\Gamma}_{z^{(m)}_v,\ z^{(m)}_{v,\ \max}}\cup[z^{(m)}_{v,\ \max},\ z^{(m)}_{v+1}]\right) \right)\cup\left(\hat{\Gamma}_{z^{(m)}_{v_0^{(m)}},\ z^{(m)}_{v_0^{(m)},\ \max}}\cup[z^{(m)}_{v_0^{(m)},\ \max},\ x_{k_m}] \right),\ & v_0^{(m)}\geq 2\nonumber
			\end{cases} 
		\end{gather}
		
\newpage

	\begin{figure}[htp]
		\centering
		\subfloat[Polygonal Line $  \tilde{\Gamma}_{x_{k_m-1},\ x_{k_m}}  $ (with red color) where $x_{k_m-1}\notin \bigcup_{i\in I}S_i\  \text{and}\  x_{k_m}\in\bigcup_{i\in I}S_i $, with $ v^{(m)}_0=1. $]{%
			\includegraphics[clip,width=0.7\columnwidth]{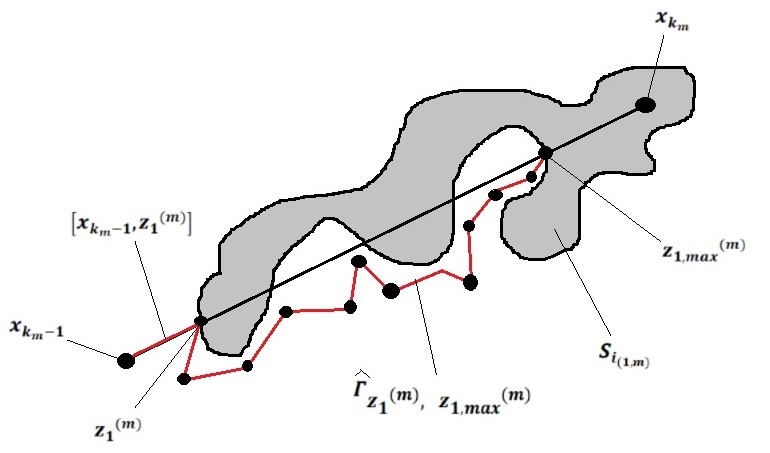}%
		}
		
		\subfloat[Polygonal Line $x_{k_m-1}\notin \bigcup_{i\in I}S_i\  \text{and}\  x_{k_m}\in\bigcup_{i\in I}S_i $, with $ v^{(m)}_0\geq 2. $]{%
			\includegraphics[clip,width=0.7\columnwidth]{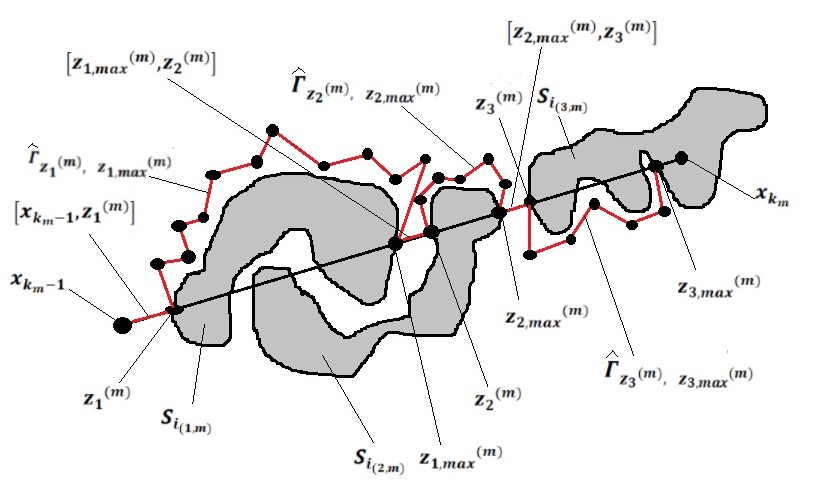}%
		}
		
		\caption{An example in $ \left(\R^d, \norm{\cdot}_2 \right)  $ that represents the Polygonal line  $  \tilde{\Gamma}_{x_{k_m-1},\ x_{k_m}}  $ (with red color) in case where $x_{k_m-1}\notin \bigcup_{i\in I}S_i\  \text{and}\  x_{k_m}\in\bigcup_{i\in I}S_i $.}
	\end{figure}
	
			\textbf{Case:} $x_{k_m-1}\notin \bigcup_{i\in I}S_i\  \text{and}\  x_{k_m}\in\bigcup_{i\in I}S_i $
	\\
	\begin{gather}
			\tilde{\Gamma}_{x_{k_m-1},\ x_{k_m}}:=\begin{cases}
				[x_{k_m-1},\ z^{(m)}_1]\cup\hat{\Gamma}_{z^{(m)}_1,\ z^{(m)}_{1,\ \max}},\ & v_0^{(m)}=1\nonumber\\
				[x_{k_m-1},\ z^{(m)}_1]\cup\left(\bigcup_{v=1}^{v_0^{(m)}-1}\left(  \hat{\Gamma}_{z^{(m)}_v,\ z^{(m)}_{v,\ \max}}\cup[z^{(m)}_{v,\ \max},\ z^{(m)}_{v+1}]\right) \right)\cup \hat{\Gamma}_{z^{(m)}_{v_0^{m}},\ z^{(m)}_{v_0,\ \max}},\ & v_0^{(m)}\geq 2 \nonumber
			\end{cases}
		\end{gather}

	\newpage

	\begin{figure}[htp]
		\centering
		\subfloat[Polygonal Line $  \tilde{\Gamma}_{x_{k_m-1},\ x_{k_m}}  $ (with red color) where $ x_{k_m-1}\in \bigcup_{i\in I}S_i\  \text{and}\  x_{k_m}\in\bigcup_{i\in I}S_i $, with $ v^{(m)}_0=1. $]{%
			\includegraphics[clip,width=0.65\columnwidth]{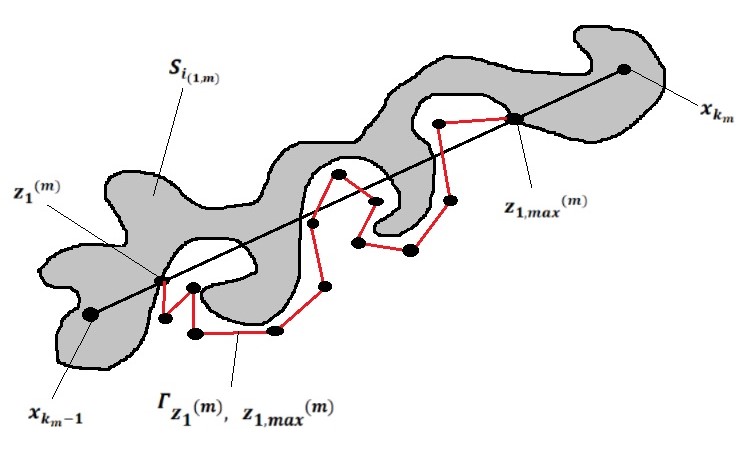}%
		}
		
		\subfloat[Polygonal Line $  \tilde{\Gamma}_{x_{k_m-1},\ x_{k_m}}  $ (with red color) where $ x_{k_m-1}\in \bigcup_{i\in I}S_i\  \text{and}\  x_{k_m}\in\bigcup_{i\in I}S_i $, with $ v^{(m)}_0\geq 2. $]{%
			\includegraphics[clip,width=0.65\columnwidth]{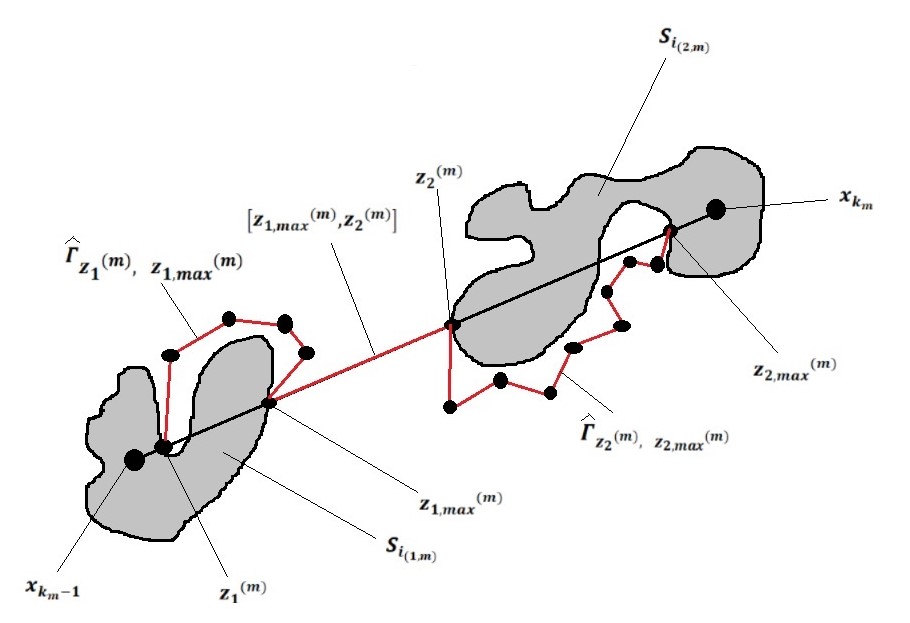}%
		}
		
		\caption{An example in $ \left(\R^d, \norm{\cdot}_2 \right)  $ that represents the Polygonal line  $  \tilde{\Gamma}_{x_{k_m-1},\ x_{k_m}}  $ (with red color) in case where $ x_{k_m-1}\in \bigcup_{i\in I}S_i\  \text{and}\  x_{k_m}\in\bigcup_{i\in I}S_i $.}
	\end{figure}

	\textbf{Case:} $ x_{k_m-1}\in \bigcup_{i\in I}S_i\  \text{and}\  x_{k_m}\in\bigcup_{i\in I}S_i $
	\begin{gather}
			\tilde{\Gamma}_{x_{k_m-1},\ x_{k_m}}:=\begin{cases}
			\hat{\Gamma}_{z^{(m)}_1,\ z^{(m)}_{1,\ \max}},\	& v_0^{(m)}=1\nonumber\\
			\bigcup_{v=1}^{v_0^{(m)}-1}\left(  \hat{\Gamma}_{z^{(m)}_v,\ z^{(m)}_{v,\ \max}}\cup[z^{(m)}_{v,\ \max},\ z^{(m)}_{v+1}]\right) \cup \hat{\Gamma}_{z^{(m)}_{v_0^{(m)}},\ z^{(m)}_{v_0^{(m)},\ \max}},\ & v_0^{(m)}\geq 2\nonumber
			\end{cases}
		\end{gather}

\newpage
	\item \textbf{Construction of Polygonal line:} $ \Gamma^{(1)}_{x_{k_m}}, $   $ \left( x_{k_m}\notin\bigcup_{i\in I}S_i \right)  $\\
	Let $ m\in\Lambda':=\left\lbrace m\in\{1,2,\dots,j_0\}:\ x_{k_m}\notin\bigcup_{i\in I} S_i\right\rbrace \setminus\tilde{\Lambda} $ 
	where 
	\begin{gather}
		\tilde{\Lambda}:=\left\lbrace m\in\{1,2,\dots,j_0\}:\ x_{k_m}\in\bigcup_{i\in I}\partial S_i\ \text{and}\ \{(1-t)x_{k_{m}-1}+ t x_{k_m}:\ t\in[0,1)\}\subset\bigcup_{i\in I}S_i\right\rbrace 
	\end{gather}
	In this case we proceed to the construction of the polygonal line  $ \Gamma^{(1)}_{x_{k_m}}. $
	\begin{claim}
	 Let $ m\in\{1,2,\dots,j_0\} $ such that $ x_{k_m}\notin\bigcup_{i\in I}S_{i}.$ Furthermore, if $ [x_{k_m},\ x_{k_m+1}]\cap\left( \bigcup_{i\in I} S_i \right)\neq\emptyset  $, then $ k_m+1\in J $ therefore, $ k_m+1=k_{m+1}. $
	\end{claim}
   \begin{proof}
	Let $x_{k_m}\notin\bigcup_{i\in I}S_{i}  $ and $ [x_{k_m},\ x_{k_m+1}]\cap\left( \bigcup_{i\in I} S_i \right)\neq\emptyset.   $ Therefore, there exists (unique) $ i^*\in I $ (since $ \bar{S}_i\cap\bar{S}_j=\emptyset,\ \forall i,j\in I,\ i\neq j $) such that $ [x_{k_m},\ x_{k_m+1}]\cap S_{i^*}\neq\emptyset. $ Consequently, there exists $ \tilde{z}\in Int(S_{i^*})\equiv S_{i^*} $ such that $ \tilde{z}\in[x_{k_m},\ x_{k_m+1}]. $ Since $ x_{k_m}\notin S_{i^*}, $ it holds that $ x_{k_m}\in\partial S_{i^*}\cup ext(S_{i^*}) $. We distinguish the following cases for $ x_{k_m} $:
	\begin{itemize}
		\item $ x_{k_m}\in \partial S_{i^*} $\\
		Then $ [x_{k_m},x_{k_m+1}]\cap\partial S_{i^*}\neq\emptyset $
		\item $ x_{k_m}\in ext(S_{i^*}) $\\
		We consider the usual curve $ \gamma_{[x_{k_m},\ x_{k_m+1}]}:[0,\ 1]\to X $ that describes the line seqment with extreme points $ x_{k_m}, x_{k_m+1}, $ (i. e $ x_{k_m}=\gamma_{[x_{k_m},\ x_{k_m+1}]}(0),\ x_{k_m+1}=\gamma_{[x_{k_m},\ x_{k_m+1}]}(1) $ and $ [x_{k_m},\ x_{k_m+1}]=\gamma_{[x_{k_m},\ x_{k_m+1}]}\left( [0,\ 1]\right)  $). Now, we consider the usual curve $ \gamma_{[x_{k_m},\ \tilde{z}]}:[0,\ t_{\tilde{z}}]\to X $ that describes the line seqment with extreme points $ x_{k_m} $ and $ \tilde{z}, $ (i. e  $ x_{k_m}=\gamma_{[x_{k_m},\ \tilde{z}]}(0) $ and $  \tilde{z}=\gamma_{[x_{k_m},\ \tilde{z}]}(t_{\tilde{z}}) $ and $ [x_{k_m},\ \tilde{z}]=\gamma_{[x_{k_m},\ \tilde{z}]}\left( [0,t_{\tilde{z}}]\right)  $. ) Since, $ x_{k_m}\in ext(S_{i^*}),\ \tilde{z}\in Int(S_{i^*}) $ from the Proposition \ref{main_prop} we receive that: 
		\begin{gather}
			\gamma_{[x_{k_m},\ \tilde{z}]}\left([0,t_{\tilde{z}}] \right) \cap\partial S_{i^*}\neq\emptyset
		\end{gather}
	From the fact that $ \gamma_{[x_{k_m},\ \tilde{z}]}\left([0,\ t_{\tilde{z}}] \right) \subset \gamma_{[x_{k_m},\ x_{k_m+1}]}\left([0,1] \right) $ we obtain that $ \gamma_{[x_{k_m},\ x_{k_m+1}]}\left([0,1] \right)\cap\partial S_{i^*}\neq\emptyset.  $ Hence, $ [x_{k_m},\ x_{k_m+1}]\cap\partial S_{i^*}\neq\emptyset. $
	\end{itemize}
From both cases, we obtain that $ k_m+1\in J, $ since $ [x_{k_m},\ x_{k_m+1}]\cap S_{i^*}\neq\emptyset $ and $ [x_{k_m},\ x_{k_m+1}]\cap \partial S_{i^*}\neq\emptyset. $
   \end{proof}
Next, we define:
	\begin{equation}
		\Gamma^{(1)}_{x_{k_m},\ (-)}:=\tilde{\Gamma}_{x_{k_m-1},\ x_{k_m}}\nonumber
		\end{equation}
	and
	\begin{equation}
		\Gamma^{(1)}_{x_{k_m},\ (+)}:=\begin{cases}
			 \bigcup_{\lambda=1}^{n-k_m}[x_{k_m+\lambda-1},\ x_{k_m+\lambda}],\ &
			 \begin{cases}
			 		[x_{k_m},\ x_{k_m+1}]\subset\left(\bigcup_{i\in I}S_i \right)^{c},\nonumber\\
			 		 m=j_0,\ n>k_m\nonumber
			 \end{cases}
		\\
			\left( \bigcup_{\lambda=1}^{k_{m+1}-(k_m+1)}[x_{k_m+\lambda-1},\ x_{k_m+\lambda}]\right)\cup \tilde{\Gamma}_{x_{k_{m+1}-1},\ x_{k_{m+1}}},\ &
			\begin{cases}
				[x_{k_m},\ x_{k_m+1}]\subset\left(\bigcup_{i\in I}S_i \right)^{c},\nonumber\\
				 m<j_0,\ n>k_m\nonumber
			\end{cases}
		\\
			\tilde{\Gamma}_{x_{k_{m+1}-1},\ x_{k_{m+1}}},\ & [x_{k_m},\ x_{k_{m}+1}] \cap\left( \bigcup_{i\in I}S_i\right) \neq\emptyset
		\end{cases}
	\end{equation}
Since $ [x_{k_m},\ x_{k_m+1}]\subset\left(\bigcup_{i\in I}S_i \right)^{c} $, it holds that $ k_{m+1}>k_m+1. $
	Finally, we define the polygonal line $ \Gamma^{(1)}_{x_{k_m}}, $ as follows:
	\begin{equation}
		 \Gamma^{(1)}_{x_{k_m}}:=\Gamma^{(1)}_{x_{k_m},\ (-)}\cup\Gamma^{(1)}_{x_{k_m},\ (+)}
	\end{equation}
	
	\item \textbf{Construction of Polygonal line:} $ \Gamma^{(2)}_{x_{k_m}}, $  $ \left( x_{k_m}\in\bigcup_{i\in I}S_i \right)  $ \\
	Let $ m\in\{1,2,\dots,j_0\} $ such that $ x_{k_m}\in\bigcup_{i\in I}S_i $. Hence, there exists (unique) $ i^*\in I $ such that $x_{k_m}\in S_{i^*}.  $ We notice that $ k_m<n, $ since if it holds that $ k_m=n, $ we obtain $ S_{i^*}\ni x_{k_m}=x_n\equiv y, $ where the last contradicts with the assumption of theorem, $ y\notin\bigcup_{i\in I}S_i. $ In this case we proceed to the construction of the polygonal line  $ \Gamma^{(2)}_{x_{k_m}}. $

	\begin{claim}
		Let $ m\in\{1,2,\dots,j_0\} $ such that $ x_{k_m}\in S_{i^*},$ for some $ i^*\in I $. Then, there exists $ \lambda^*\in\N $ such that $ \left[x_{k_m+\lambda^*-1},\ x_{k_m+\lambda^*} \right]\cap \partial S_{i^*}\neq\emptyset.  $
	\end{claim}
	\begin{proof}
			  We notice that $ x_{k_m}\in S_{i^*}  $ and $ y\in\partial S_{i^*}\cup ext(S_{i^*}). $ Next, we distinguish the following cases for $ y. $
			\begin{itemize}
				\item[-] $ y=x_n\in\partial S_{i^*} $\\
				Therefore, $ [x_{n-1},x_n] \cap\partial S_{i^*}$ where $ x_{n-1}=x_{k_m+\lambda^*-1} $ for some (unique) $ \lambda^*\in\N $ with $ \lambda^*\geq 3. $
				\item[-] $ y\in ext(S_{i^*}) $\\
				We consider the polygonal line $ \gamma_{x_{k_m},y}:\tilde{I}\to X $ with $ \gamma_{x_{k_m},y}:=\bigcup_{r=1}^{n-k_m}[x_{k_m+r-1},\ x_{k_m+r}] $ that connects the points $ x_{k_m} $ and $ y. $ Since $ \gamma_{x_{k_m},y} $ is a continuous curve and $ x_{k_m}\in S_{i^*},\ y\in ext(S_{i^*})  $, from Proposition \ref{main_prop}, it holds that $ \gamma_{x_{k_m},y}(\tilde{I})\cap\partial S_{i^*}\neq\emptyset. $ Consequently, there exists $ t_0\in \tilde{I}, $ such that $ \gamma_{x_{k_m},y}(t_0)\in[x_{k_m+\lambda-1},\ x_{k_m+\lambda}]\cap\partial S_{i^*} $ for some $ \lambda\in\{1,\dots,n-k_m\}. $
			\end{itemize}
			From both cases is concluded that there exists $ \lambda\in\N $ such that $ [x_{k_m+\lambda-1},\ x_{k_m+\lambda}]\cap\partial S_{i^*}\neq\emptyset. $
		\end{proof}

Next we define: 
\begin{gather}
	\lambda_0^{(m)}:=\max\left\lbrace \lambda\in\left\lbrace 1,2,\dots,n-k_m\right\rbrace:\ \left[ x_{k_m+\lambda-1},\ x_{k_m+\lambda}\right]\cap\partial S_{i^*}\neq\emptyset\right\rbrace\nonumber\\
	t_{\lambda_0^{(m)}}^*:=\min\left\lbrace t\in[0,1]:\ (1-t) x_{k_m+\lambda_0^{(m)}-1}+t\ x_{k_m+\lambda_{0}^{(m)}}\in\partial S_{i^*}\right\rbrace\nonumber\\
	z_{\lambda_{0}^{(m)}}^*:=\left( 1-t_{\lambda_{0}^{(m)}}^*\right)  x_{k_m+\lambda_0^{(m)}-1}+t_{\lambda_0^{(m)}}^*\ x_{k_m+\lambda_0^{(m)}}\nonumber
\end{gather}

In case where $ \lambda_0^{(m)}<n-k_m, $ we define
	\begin{gather}
	H:=\left\lbrace \lambda\in\left\lbrace 1,2,\dots,n-\left( k_m+\lambda_0^{(m)}\right) \right\rbrace:\ k_m+\lambda_0^{(m)}+\lambda\in J\right\rbrace 
	\end{gather}
	
Finally, we define the Polygonal line $ \Gamma^{(2)}_{x_{k_m}}, $ as follows:
\begin{equation}
	\Gamma^{(2)}_{x_{k_m}}:=\begin{cases}
		\hat{\Gamma}_{z^{(m)}_{v_0^{(m)},\ \max},\ z^*_{\lambda_0^{(m)}} }, & k_{m}+\lambda_0^{(m)}\in J\nonumber\\
		
				\hat{\Gamma}_{z^{(m)}_{v_0^{(m)},\ \max},\ z^*_{\lambda_0^{(m)}} }\bigcup \left[z^*_{\lambda_0^{(m)}},\ x_{k_m+\lambda_0^{(m)}} \right], & \begin{cases}
						k_m+\lambda_0^{(m)}\notin J,\ H=\emptyset,\nonumber\\ n=k_m+\lambda_0^{(m)}\nonumber
					\end{cases}
					\\
				\hat{\Gamma}_{z^{(m)}_{v_0^{(m)},\ \max},\ z^*_{\lambda_0^{(m)}} }\bigcup \left[z^*_{\lambda_0^{(m)}},\ x_{k_m+\lambda_0^{(m)}} \right]  \bigcup_{\lambda=1}^{n-(k_m+\lambda_0^{(m)})} \left[x_{k_m+\lambda_0^{(m)}+(\lambda-1)},\ x_{k_m+\lambda_0^{(m)}+\lambda} \right] ,\ & \begin{cases}
						k_m+\lambda_0^{(m)}\notin J,\ H=\emptyset,\nonumber\\ n>k_m+\lambda_0^{(m)}\nonumber
					\end{cases}
				\\
				\hat{\Gamma}_{z^{(m)}_{v_0^{(m)},\ \max},\ z^*_{\lambda_0^{(m)}} }\bigcup \left[z^*_{\lambda_0^{(m)}},\ x_{k_m+\lambda_0^{(m)}} \right]  , & \begin{cases}
					k_{m}+\lambda_0^{(m)}\notin J,\nonumber \\ 
					H\neq\emptyset,\ \min H=1\nonumber
				\end{cases} 
			\\
				\hat{\Gamma}_{z^{(m)}_{v_0^{(m)},\ \max},\ z^*_{\lambda_0^{(m)}} }\bigcup \left[z^*_{\lambda_0^{(m)}},\ x_{k_m+\lambda_0^{(m)}} \right]  \bigcup_{\lambda=1}^{\min H-1} \left[x_{k_m+\lambda_0^{(m)}+(\lambda-1)},\ x_{k_m+\lambda_0^{(m)}+\lambda} \right], & \begin{cases}
				k_{m}+\lambda_0^{(m)}\notin J,\nonumber\\
				 H\neq\emptyset,\ \min H>1,\nonumber
			\end{cases}
		\end{cases}
\end{equation}

\end{itemize}

-----------------------------------------------------------------------------

Finally, we define the polygonal line $ \Gamma_{x,y} $ that connects the points $ x=x_1 $ and $ y=x_n. $
\begin{equation}
	\Gamma_{x,y}:=\begin{cases}
		\Gamma^*,\ & k_1=2\nonumber\\
		\left( \bigcup_{l=2}^{k_1-1}[x_{l-1},x_{l}]\right) \cup\Gamma^*,\ &k_1\geq 3
	\end{cases}
\end{equation}
where $ \Gamma^*:=\bigcup_{l\in\Lambda} \Gamma_{x_{k_l}}$ with 
\begin{equation}
	\Gamma_{x_{k_l}}:=\begin{cases}
		\Gamma^{(1)}_{x_{k_l}},\ & x_{k_l}\notin\bigcup_{i\in I}S_i\nonumber\\
		\Gamma^{(2)}_{x_{k_l}},\ & x_{k_l}\in\bigcup_{i\in I}S_i\nonumber
	\end{cases}
\end{equation}
and $ \Lambda:=\{1,2,\dots,j_0\}\setminus\left( \tilde{\Lambda}\cup\Lambda^*\right),  $
where 
\begin{gather}
	\Lambda^*:=\left\lbrace m\in\{1,2,\dots, j_0\}:\ x_{k_m-1}\in\bigcup_{i\in I}\partial S_i\ \text{and}\ \{(1-t)x_{k_m-1}+t x_{k_m}:\ t\in(0,1]\}\subset\bigcup_{i\in I} S_i \right\rbrace\nonumber.
\end{gather}
	\end{proof}

	\newpage

	\begin{tabular}{l}
		Savvas Andronicou\\ University of Cyprus \\ Department of Mathematics \& Statistics \\ P.O. Box 20537\\
		Nicosia, CY- 1678 CYPRUS
		\\ {\small \tt andronikou.savvas@ucy.ac.cy}
	\end{tabular}
	\begin{tabular}{lr}
		Emmanouil Milakis\\ University of Cyprus \\ Department of Mathematics \& Statistics \\ P.O. Box 20537\\
		Nicosia, CY- 1678 CYPRUS
		\\ {\small \tt emilakis@ucy.ac.cy}
	\end{tabular}
	
 \end{document}